\documentclass[11pt,a4paper]{article}
\usepackage[english]{babel}
\usepackage{amsmath,amsthm,amssymb,bbm} 
\usepackage{booktabs}
\newtheorem{theorem}{Theorem}
\newtheorem{lemma}{Lemma}
\newtheorem{corollary}{Corollary}

\theoremstyle{definition}
\newtheorem{remark}{Remark}

\DeclareSymbolFont{ugrf@m}{U}{eur}{m}{n}
\DeclareMathSymbol{\updelta}{\mathord}{ugrf@m}{"0E}

\begin{document}
	\sloppy

		\title{Needles and straw in a haystack: robust  confidence for possibly sparse sequences} 
\author{Eduard Belitser and Nurzhan Nurushev\\ {\it VU Amsterdam and University of Amsterdam}}
\date{}
\maketitle		
	
		\begin{abstract}
			In the general \emph{signal+noise} (allowing non-normal, non-independent observations) model  
			we construct an empirical Bayes posterior 
			which we then use for \emph{uncertainty quantification} for the unknown, possibly sparse,
			signal. We introduce a novel \emph{excessive bias restriction} (EBR) condition, which
			gives rise  to a new slicing of the entire space that is suitable for uncertainty quantification.
			Under EBR and some mild  \emph{exchangeable exponential moment condition} on the noise, 
			we establish the local (oracle) optimality of the proposed confidence ball. 
			Without EBR, we  
			derive the full coverage for confidence balls of at least $\sigma n^{1/4}$-radius, implying 
			the local optimality only for cases when the oracle rate is at least of the order $\sigma n^{1/4}$.
			In passing, we also get the local optimal results for estimation and posterior contraction 
			problems. Adaptive minimax results  (also for the estimation and posterior contraction 
			problems) over various sparsity classes follow from our local results. 
		\end{abstract}
\vspace{-20pt}
\footnote{\hspace{-3mm}
	\textit{MSC2010 subject classification:}
	primary 62G15,  secondary 62C12.
	
}

	\section{Introduction}
	\paragraph{The model and the main problem.} 
	Suppose we observe  $X=X^{(\sigma,n)}=(X_1,\ldots,X_n)$:
	\begin{eqnarray}
		\label{model}
		X_i=\theta_i+ \sigma \xi_i, \quad i\in [n]=\{1,\ldots,n\}, 
	\end{eqnarray}
	where $\theta=(\theta_1,\ldots, \theta_n) \in \mathbb{R}^n$ 
	is an unknown high-dimensional parameter of interest, the $\xi_i$'s are 
	random errors with $\mathrm{E} \xi_i =0$, $\mbox{Var}(\xi_i)\le C_\xi$,
	$\sigma>0$ is the known noise intensity. 
	The goal is to make inference about the parameter $\theta$ based on 
	the data $X$: recovery of  $\theta$ and \emph{uncertainty quantification} by constructing 
	an \emph{optimal confidence set}. 
	We pursue \emph{robust inference} in the sense that 
	the distribution of the error vector $\xi=(\xi_1,\ldots,\xi_n)$ 
	is unknown and can also depend on $\theta$, 
	but assumed to satisfy only certain mild \emph{exchangeable exponential moment condition}; 
	see Condition \eqref{cond_nonnormal} in Section \ref{section_preliminaries}.
	For inference on $\theta$, we exploit the empirical Bayes approach. 
	We derive non-asymptotic results, which imply asymptotic assertions as well if needed. 
	Possible asymptotic regimes are decreasing noise level $\sigma \to 0$, 
	high-dimensional setup $n \to \infty$ (the leading case 
	for high dimensional models), or their combination, e.g., $\sigma = n^{-1/2}$ and $n\to \infty$.

	Useful inference is not possible without some structure on the parameter $\theta$. 
	Popular structural assumptions are \emph{smoothness} and \emph{sparsity},
	in this paper we are concerned with the latter.
	The best studied problem in the sparsity context is that of estimating $\theta$
	in the many normal means model,  a variety of estimation methods 
	and results are available in the literature:  \cite{Donoho&Johnstone:1994b}, 
	\cite{Birge&Massart:2001}, \cite{Johnstone&Silverman:2004},
	\cite{Abramovich&Grinshtein&Pensky:2007}, \cite{Castillo&vanderVaart:2012}, 
	\cite{vanderPas&Kleijn&vanderVaart:2014}. However, even an optimal 
	estimator does not reveal how far it is from $\theta$. It is of importance 
	to quantify this uncertainty, which can be seen as the problem of constructing 
	confidence sets for $\theta$.
	
	\paragraph{Bayesian approach and accompanying posterior contraction problem.}
	Many inference methods have Bayesian connections. For example, even some seemingly non Bayesian 
	estimators can be obtained as certain quantities (like posterior mode for penalized minimum 
	contrast estimators) of the (empirical Bayes) posterior distributions  resulting from imposing 
	some specific priors on the parameter; cf.\ \cite{Johnstone&Silverman:2004} and 
	\cite{Abramovich&Grinshtein&Pensky:2007}. 
	Although the Bayesian methodology is used or can be related to 
	in constructing many (frequentist) inference procedures, only recently the posterior 
	distributions themselves have been studied 
	in the sparsity context: \cite{Castillo&vanderVaart:2012}, \cite{vanderPas&Kleijn&vanderVaart:2014},
	\cite{Martin&Walker:2014}, \cite{Castilloetal:2015}, \cite{Bhattacharya&etc:2016}, 
	\cite{Rousseau&Szabo:2017}, \cite{Rockova:2018}.  
	
	In this paper, for inference on $\theta$ we use 
	an empirical Bayes approach. Since any Bayesian approach always delivers a posterior
	$\pi(\vartheta|X)$ (in the posteriors for $\theta$, we will use the variable $\vartheta$ 
	to distinguish it from the ``true''  $\theta$), 
	an accompanying problem of interest is  
	the  contraction of the resulting (empirical Bayes) posterior to the ``true'' $\theta$ 
	from the frequentist perspective of the ``true'' measure $\mathrm{P}_\theta$, 
	the distribution of $X$ from \eqref{model}.
	The quality of posterior is characterized by the posterior contraction rate.
	We pursue a novel local approach by allowing the posterior contraction rate to be a local quantity, i.e., depending on the true $\theta$, whereas global minimax rates 
	are typically studied in the literature on Bayesian nonparametrics.

	A common Bayesian way to model sparsity structure is by the so called 
	two-groups priors. Such a prior puts positive mass on vectors $\theta$ with some exact zero 
	coordinates (zero group) and the remaining coordinates (signal group) are 
	drawn from a chosen distribution.
	So the marginal prior for each coordinate is a mixture of 
	a continuous distribution and a point-mass at zero. 
	In \cite{Castillo&vanderVaart:2012} it is shown that for a suitably chosen two-groups prior, the posterior
	concentrates around the true $\theta$ at the minimax rate (as $n\to \infty$) 
	for two sparsity classes, 
	\emph{nearly black vectors} $\ell[p_n]$ with $p_n$ nonzero coordinates 
	and \emph{weak $\ell_s$-balls} $m_s[p_n]$.  
	As pointed out  by \cite{Castillo&vanderVaart:2012} (also by 
	\cite{Johnstone&Silverman:2004}), the distributions of non-zero coordinates 
	should not have too light tails, otherwise
	one gets sub-optimal rates. The important Gaussian case is
	for example excluded. This has to do with the so called 
	\emph{over-shrinkage effect} of the normal prior with a fixed mean for 
	nonzero coordinates, which pushes the posterior too much towards
	the prior mean, 
	missing the true parameter that in general differs from the prior mean. 
	That is why \cite{Johnstone&Silverman:2004} and \cite{Castillo&vanderVaart:2012} 
	discard normal priors on non-zero coordinates and use heavy tailed priors.
	A way to construct such a prior is to put a next level heavy-tailed 
	prior, like  half-Cauchy, on the variance in the normal prior, resulting 
	in the so called (one-component) horseshoe prior on $\theta$ (cf.\  
	\cite{Carvalho&Polson&Scot:2010} and \cite{vanderPas&Kleijn&vanderVaart:2014}). 
	In the present paper we show that normal priors are still usable (cf.\ \cite{Martin&Walker:2014}) 
	and even lead to strong local results (and even for non-normal models) 
	if combined with empirical Bayes approach.

	\paragraph{Uncertainty quantification problem.}
	The main aim in this paper is to construct confidence sets with optimal properties.
	The size of a confidence set is measured by the smallest radius of 
	a ball containing this set, hence it suffices to consider confidence balls. 
	For the usual norm $\| \cdot \|$ in $\mathbb{R}^n$, a random ball in $\mathbb{R}^n$ is 
	$B(\hat{\theta},\hat{r})=\{\theta \in \mathbb{R}^n: \|\hat{\theta}-\theta\| \le \hat{r}\}$,
	where the center $\hat{\theta}=\hat{\theta}(X):\mathbb{R}^n \mapsto \mathbb{R}^n$ and  radius 
	$\hat{r}=\hat{r}(X): \mathbb{R}^n \mapsto \mathbb{R}_+ =[0,+\infty]$ are
	measurable functions of the data $X$. 
	Let us introduce the optimality framework for uncertainty quantification. 
	The goal is to construct such a confidence ball $B(\hat{\theta},C\hat{r})$
	that for any $\alpha_1,\alpha_2\in(0,1]$
	and some functional $r(\theta)=r_{\sigma,n}(\theta)$,
	$r:\mathbb{R}^n\rightarrow \mathbb{R}_+$, there exist $C, c > 0$ such that
	\begin{align}
		\label{defconfball}
		\sup_{\theta\in\Theta_0}\mathrm{P}_\theta\big(\theta\notin 
		B(\hat{\theta},C\hat{r})\big)\le\alpha_1,\quad \sup_{\theta\in\Theta_1}
		\mathrm{P}_\theta\big(\hat{r}\ge c r(\theta)\big)\le\alpha_2,
	\end{align}
	for some $\Theta_0,\Theta_1\subseteq\mathbb{R}^n$. The function $r(\theta)$, 
	called the \emph{radial rate}, is a benchmark for the effective radius of 
	the confidence ball $B(\hat{\theta},C\hat{r})$.
	The first expression in \eqref{defconfball} is called \emph{coverage relation} 
	and the second \emph{size relation}. 
	Notice that our approach is local (and hence genuinely adaptive) as 
	the radial rate $r(\theta)$ is a function of the ``true'' parameter $\theta$.
	Recall the common (global) minimax adaptive version of \eqref{defconfball}:
	given a family of sets $\Theta_\beta$ with corresponding minimax estimation rates 
	$r(\Theta_\beta)$ indexed by a structural parameter $\beta\in\mathcal{B}$ 
	(e.g., smoothness or sparsity), the minimax adaptive version of  \eqref{defconfball} would 
	be obtained by taking $\Theta_0=\Theta_1=\Theta_\beta$ 
	and the radial rate $r(\theta)=r(\Theta_\beta)$ for all $\theta\in\Theta_\beta$ and all
	$\beta \in\mathcal{B}$.
	
	Coming back to our local framework \eqref{defconfball}, it is desirable to find the 
	smallest  $r(\theta)$ and the biggest $\Theta_0,\Theta_1$, 
	for which \eqref{defconfball} holds. These are contrary requirements, 
	so we have to trade them off against each other.
	There are different ways of 
	doing this, leading to different optimality frameworks. For example, if we insist on 
	overall uniformity $\Theta_0=\mathbb{R}^n$, then the results in \cite{Li:1989} and 
	\cite{Cai&Low:2004} (more refined versions are in \cite{Baraud:2004} 
	and \cite{Nickl&vandeGeer:2013})
	say basically that the radial rate $r$ cannot be of a faster order than $\sigma n^{1/4}$ for every 
	$\theta$ and is at least of order $\sigma n^{1/2}$ for some $\theta$.  
	This means that  any confidence ball that is optimal with respect to the optimality 
	framework \eqref{defconfball}  with $\Theta_0=\mathbb{R}^n$  
	will necessarily have a big size, even if the true $\theta$ happens to lie in a very ``good'', smooth or sparse, 
	class $\Theta_1$. Many good confidence sets cannot be optimal in this sense 
	(called ``honest'' in some papers)  and effectively excluded from the consideration. 
	For minimax adaptive versions of \eqref{defconfball} this means that as soon 
	as we require $\Theta_0=\Theta_\beta$, 
	$\beta\in\mathcal{B}$ in the coverage relation, the minimax rate 
	$r(\Theta_\beta)$ in the size relation is unattainable even for $\beta \in \mathcal{B}=
	\{ \beta_1,\beta_2\}$; 
	cf.\ \cite{Nickl&vandeGeer:2013} for two nearly black classes.
	Essentially, the overall uniform coverage and optimal size properties can not hold 
	together,  it is necessary to sacrifice at least one of these, preferably as little as possible.
	We argue that it is unreasonable to pursue an optimality framework with 
	the entire space $\Theta_0=\mathbb{R}^n$ in the coverage relation, 
	because this leads to discarding many good procedures and optimality of uninteresting ones. 
	Instead, it makes sense to sacrifice in the set $\Theta_0=\mathbb{R}^n \backslash \Theta'$, 
	by removing a preferably small portion of  ``deceptive parameters''  $\Theta'$ from 
	$\mathbb{R}^n$ so that that  the optimal radial rates become attainable in the size relation with 
	interesting (preferably ``massive'') sets $\Theta_1$. 
	
	This ``deceptiveness'' phenomenon is well understood for some smoothness structures (e.g., Sobolev scale), 
	especially in global minimax settings; see \cite{Robins&vanderVaart:2006}, \cite{Bull&Nickl:2013}, \cite{Belitser:2017} and \cite{Szabo&vanderVaart&vanZanten:2015}.
	If we now insist on the optimal size property in \eqref{defconfball} for all $\Theta_\beta$, $\beta\in\mathcal{B}$,
	the coverage relation in  \eqref{defconfball} will not hold for all $\Theta_0=\Theta_\beta$, but  only
	for $\Theta_0 = \Theta_\beta \backslash \Theta'$, with some set of 
	``deceptive parameters''  $\Theta'$ removed from $\Theta_\beta$.
	In \cite{Szabo&vanderVaart&vanZanten:2015} such parameters are called
	``inconvenient truths'' and an implicit construction of a $\theta'\in\Theta'$ is given. 
	Examples of non-deceptive parameters are the set of \emph{self-similar} parameters 
	$\Theta_0=\Theta_{ss}$  introduced by  \cite{Picard&Tribouley:2000} and studied by 
	\cite{Bull:2012}, \cite{Bull&Nickl:2013}, 
	\cite{Szabo&vanderVaart&vanZanten:2015}, and the set of \emph{polished tail parameters} 
	$\Theta_0=\Theta_{pt}$ considered by \cite{Szabo&vanderVaart&vanZanten:2015}. 
	In all the above mentioned papers global minimax radial rates 
	(i.e., $r(\theta) = r(\Theta_\beta)$ for all $\theta \in \Theta_\beta$)
	for specific smoothness structures are studied.
	A local approach, delivering also the adaptive minimax results for many smoothness 
	structures simultaneously, is considered by \cite{Babenko&Belitser:2010} 
	for posterior contraction rates and by \cite{Belitser:2017} for constructing optimal 
	confidence balls. In \cite{Belitser:2017}, yet a more general (than $\Theta_{ss}$ and $\Theta_{pt}$)
	set of non-deceptive parameters was introduced,
	$\Theta_0=\Theta_{ebr}$, parameters satisfying the so called 
	\emph{excessive bias restriction} (EBR).  More on this can be found in Section \ref{subsec_ebr}.
	
	To the best of our knowledge, there are very few papers about adaptive results on 
	uncertainty quantification \eqref{defconfball}. The case of two nearly black classes 
	is treated by \cite{Nickl&vandeGeer:2013}, the ``general polished tail'' condition was 
	introduced in \cite{Rousseau&Szabo:2016} to describe non-deceptive parameters, 
	a restricted scale of nearly black classes is treated in \cite{vanderPas&Szabo&vanderVaart:2017},
	where a version of our EBR condition is used, more on relation to paper  \cite{vanderPas&Szabo&vanderVaart:2017} 
	can be found in Section \ref{subsec_duscussion}.
	
	\paragraph{The scope of this paper.}
	In this paper, we introduce a family of normal mixture priors and propose an empirical Bayes procedure
	(in fact, two procedures). We use the normal likelihood, whereas the true model \eqref{model} 
	does not have to be normal (and independence of $\xi_i$'s is not required either), but only satisfying 
	some mild Condition (A1) (called {\it exchangeable exponential moment condition}). 
	There are three distinctive features of our approach: \emph{robust}, \emph{local} and \emph{refined}.
	
	First, \emph{robust} means that our results cover also misspecified models, as
	we allow the $\xi_i$'s to be not necessarily independent normals (a certain type of 
	error misspecification was also mentioned in a remark of the supplement to 
	\cite{Castilloetal:2015}), 
	but only satisfying  Condition (A1).
	It turned out that, although we use the normal likelihood in the Bayesian analysis, 
	in the proof of the main results we can handle the frequentist behavior 
	of the posterior from the perspective of the true measure only on the basis of Condition (A1).
	
	Second, we develop the novel \emph{local} approach, meaning that the radial rate $r(\theta)$ in \eqref{defconfball} 
	is allowed to be a function of $\theta$, which, in a way, measures the 
	amount of sparsity for each $\theta\in\mathbb{R}^n$: the smaller $r(\theta)$, 
	the more sparse $\theta$. 
	The local radial rate $r(\theta)$ is constructed as the best (smallest) rate over 
	a certain family of local rates, therefore called \emph{oracle rate}. 
	We demonstrate that the local approach is more powerful than global in that
	we do not need to impose any specific sparsity structure,  
	because the proposed  local approach automatically exploits the ``effective'' 
	sparsity of each underlying $\theta$, and our local results imply a whole panorama of
	the global minimax results for many scales at once. More on this is in Section \ref{implication}.
	
	Third, we derive the local posterior contraction result for the resulting empirical Bayes posterior $\hat{\pi}(\vartheta |X)$
	in the \emph{refined non-asymptotic formulation}: 
	$\sup_{\theta\in\mathbb{R}^n}\mathrm{E}_{\theta}\hat{\pi}(\|\vartheta -\theta\|^2 
	\ge M_0 r^2(\theta) + M \sigma^2|X) \le H_0e^{-m_0 M}$ 
	for some fixed $M_0, H_0, m_0>0$ and arbitrary $M \ge 0$,
	as an exponential non-asymptotic concentration bound in terms of $M$, uniformly 
	in $\theta\in \mathbb{R}^n$. This formulation provides  a rather subtile 
	characterization of the quality of the posterior    
	(finer, than, e.g., asymptotically in terms of the dimension $n$), allowing 
	subtle analysis for various asymptotic regimes. 
	This result is of interest and importance on its own as it actually establishes the contraction 
	of the empirical Bayes posterior with the local rate $r(\theta)$.
	Besides, we  obtain the oracle estimation result (also in similar refined formulation, 
	finer than traditional oracle inequalities) by constructing 
	an estimator, the empirical Bayes posterior mean, which converges to $\theta$ with the local 
	rate $r(\theta)$. 
	This result, besides being an ingredient for the uncertainty quantification problem \eqref{defconfball}, 
	is also of interest and importance on its own as it delivers the same (oracle and minimax) estimation  
	results as in \cite{Abramovich&Grinshtein&Pensky:2007} and \cite{Johnstone&Silverman:2004} 
	and posterior convergence results as in \cite{Castillo&vanderVaart:2012}, 
	obtained for different priors.
	

	Next, we construct a confidence ball 
	by using the empirical Bayes posterior quantities. 
	Since we want the size of our confidence sets to be of an oracle rate order, 
	this comes with the price that the coverage property can hold
	uniformly only over some set of parameters satisfying the so called \emph{excessive bias restriction} 
	(EBR) 
	$\Theta_0=\Theta_{\rm eb}\subseteq\mathbb{R}^n$. 
	The main result consists in establishing the optimality \eqref{defconfball} 
	of the constructed confidence ball for the optimality framework 
	$\Theta_0=\Theta_{\rm eb}$, $\Theta_1=\mathbb{R}^n$ and the local radial rate 
	$r(\theta)$. 
	The important consequence of our local approach is that a whole panorama 
	of adaptive (global) minimax results (for all the three problems: estimation, 
	posterior contraction rate and confidence sets)  over \emph{all} sparsity scales 
	\emph{covered} by $r(\theta)$ (see Section \ref{implication}) follow from our local results. 
	In particular, our local results imply the same type of  
	adaptive minimax  estimation results over sparsity scales as in \cite{Johnstone&Silverman:2004}, 
	and the same type of global minimax results on contraction posterior rates 
	as in \cite{Castillo&vanderVaart:2012} (and actually more).
	
	Also we treat the situation when $\Theta_0= \mathbb{R}^n$ in \eqref{defconfball}
	by constructing a confidence ball such that its radius is 
	of the order $\sigma n^{1/4} + r(\theta)$. As we already discussed, the term 
	$\sigma n^{1/4}$ in the size relation is necessary for the uniform coverage to hold. 
	Clearly, this confidence ball will have optimal size only 
	for non-sparse parameters (for which $r(\theta) \ge c\sigma n^{1/4}$).

	Although the original motivation of the EBR condition was to remove the deceptive parameters, 
	it turned out to be a very useful notion in the context of uncertainty quantification. 
	In effect, the EBR condition leads to a \emph{new sparsity EBR-scale} 
	which gives the slicing of the entire space that is very suitable for uncertainty quantification.
	This provides a new perspective at the above mentioned ``deceptiveness'' issue: 
	basically, each parameter is  deceptive (or non deceptive) to some extent.
	It is the structural parameter of the new EBR-scale that 
	measures the deceptiveness amount, and the (mild and controllable) price for handling 
	deceptive parameters is the effective amount of inflating of the confidence ball 
	that matches the amount of deceptiveness needed to provide a high coverage. 
	More on the EBR condition can be found in Section \ref{subsec_ebr}.

	The paper is organized as follows. In Section \ref{section_prelim} we introduce the notation, 
	the prior, describe the empirical Bayes procedure in detail, make a link with the penalization method, 
	and provide some conditions. 
	Section $\ref{main_results}$, where we also introduce the EBR, contains the main 
	results of the paper. In Section $\ref{sec_discussion_and_EBR}$, we discuss some 
	variations of our results, present concluding remarks and discuss the EBR. 
	The theoretical results are illustrated in Section $\ref{simulations}$ 
	by a small simulation study. The proofs of the lemmas and theorems are 
	given in Sections $\ref{proofs_lemmas}$ and $\ref{proofs_theorems}$ respectively.

	\section{Preliminaries}
	\label{section_preliminaries}
	First we introduce some notation and a certain family of normal priors (similar to priors from 
	\cite{Belitser:2017} but geard towards modeling sparsity rather than smoothness). Next,  by applying 
	the empirical Bayes approach to the the normal likelihood, we derive an empirical 
	Bayes posterior which we will use in the construction of the estimator and the confidence ball. 
	The empirical Bayes procedure is linked to the penalization method.
	We complete this section with some conditions on the errors $\xi_i$'s and the prior.  
	
	\label{section_prelim}
	\subsection{Notation}
	Denote the  probability measure of $X$ from the model (\ref{model}) by 
	$\mathrm{P}_\theta=\mathrm{P}^{(\sigma,n)}_\theta$, and by $\mathrm{E}_\theta$ 
	the corresponding expectation. 
	For notational simplicity, we often skip the 
	dependence on $\sigma$ and $n$. 
	Denote by $\mathbbm{1}_E=\mathbbm{1}\{E\}$ the indicator function of the event $E$, 
	by $|S|$ the cardinality of the set $S$, the difference of sets 
	$S\backslash S_0 =\{s \in S:\,  s\not\in S_0\}$.
	Let $[k]=\{1,\ldots,k\}$ and $[k]_0=\{0\} \cup [k]$ 
	for $k\in\mathbb{N}=\{1,2,\ldots\}$.  
	For $I\subseteq [n]$,  define $I^c=[n] \backslash I$. 
	Let $\mathcal{I}=\mathcal{I}_n=\{I: \, I \subseteq [n]\}$
	be the family of all subsets of $[n]$ including the empty set.
	If the summation range in $\sum_I$ is not specified (for brevity), this means $\sum_{I\in \mathcal{I}}$.  
	Throughout we assume the conventions that $\sum_{i\in \varnothing}a_i=0$,  
	$\sum_a^b a_i=\sum_{a\le i \le b} a_i$
	for any $a_i,a,b\in\mathbb{R}$ and
	$0\log(c/0)=0$ (hence $(c/0)^0=1$) for any $c>0$. 
	Let $\theta_{(1)}^2\le\theta_{(2)}^2\le\ldots \le \theta_{(n)}^2$ and 
	$\theta_{[1]}^2\ge\theta_{[2]}^2\ge\ldots \ge \theta_{[n]}^2$ 
	be the ordered values of $\theta_{1}^2,\ldots, \theta_{n}^2$.
	To have some quantity well defined in the sequel, introduce also 
	$0=\theta^2_{(0)}=\theta_{[n+1]}^2$ and $\theta_{[0]}^2=\theta^2_{(n+1)}=\infty$.

	If random quantities  appear in a relation, this relation should be understood 
	in $\mathrm{P}_{\theta}$-almost sure sense. 
	Throughout $\phi(x,\mu,\sigma^2)$ will be the density 
	of $\mu+\sigma Z\sim \mathrm{N}(\mu,\sigma^2)$ at point $x$, where 
	$Z\sim \mathrm{N}(0,1)$. 
	By convention, $\mathrm{N}(\mu,0)=\delta_\mu$ denotes a Dirac measure at point $\mu$.
	The symbol $\triangleq$ will refer to equality by definition, $(a\vee b) = \max\{a,b\}$ and
	$(a\wedge b) = \min\{a,b\}$. Finally,  denote $X(I)=(X_i\mathbbm{1}\{i\in I\}, i\in\mathbb{N}_n)$ 
	for $I \in \mathcal{I}$, and let $\langle x, y \rangle = \sum_i x_i y_i $ denote the usual scalar 
	product between $x,y \in \mathbb{R}^n$.

	\subsection{Multivariate normal prior}
	When deriving all the posterior quantities in the Bayesian analysis below, 
	we will use the normal likelihood $\ell(\theta, X) =(2\pi \sigma^2)^{-n/2}
	\exp\{-\|X-\theta\|^2/2\sigma^2\}$, which is equivalent to imposing  
	the classical high-dimensional normal model  $X=(X_i, i\in\mathbb{N}_n)\sim
	\bigotimes_{i=1}^n \mathrm{N}(\theta_i, \sigma^2)$. Recall however that 
	the ``true'' model $X \sim \mathrm{P}_{\theta}$ is not assumed to be normal, 
	but satisfying Condition \eqref{cond_nonnormal}.

	To model possible sparsity in the parameter $\theta$, the coordinates of $\theta$ 
	can be split into two distinct groups of coordinates of $\theta$: for some 
	$I \in\mathcal{I}$,  $\theta_I = (\theta_i, i \in I)$ and $\theta_{I^c}=(\theta_i, i \in I^c)$, 
	so that $\theta=(\theta_I, \theta_{I^c})$. The group of coordinates 
	$\theta_{I^c}=(\theta_i, i \not\in I)$  consists of (almost) zeros 
	and  $\theta_{I}=(\theta_i, i \in I)$ 
	is the group of non-zeros coordinates. For any $\theta\in\mathbb{R}^n$ 
	(even ``not sparse'' one) there is the best (oracle) splitting into two groups, 
	we will come back to this in Section \ref{main_results}.
	To model sparsity, we propose a prior on $\theta$ given $I$ as follows:
	\begin{align}
		\label{gen_norm_prior}
		\pi_I&=\bigotimes\nolimits_{i=1}^n \mathrm{N}\big(\mu_i(I), \tau_i^2(I)\big), \quad
		\mu_i(I)=\mu_i\mathbbm{1}\{i\in I\},  \quad
		\tau^2_i(I)=\sigma^2K_n(I)\mathbbm{1}\{i \in I\},
	\end{align}
	and $K_n(I)=(\tfrac{en}{|I|}-1)\mathbbm{1} \{I\not = \varnothing\}$.
	The indicators in prior \eqref{gen_norm_prior} ensure the sparsity  of the group $I^c$.
	The rather specific choice of $K_n(I)$ is made for the sake of concise expressions 
	in later calculations, many other choices are actually possible.
	By using normal likelihood $\ell(\theta, X) =(2\pi \sigma^2)^{-n/2}\exp\{-\|X-\theta\|^2/2\sigma^2\}$, 
	the corresponding posterior distribution for $\theta$ is readily obtained:
	\begin{align}
		\label{norm_poster_I}
		\pi_I(\vartheta|X) = \bigotimes_{i=1}^n 
		\mathrm{N}\Big(\frac{\tau_i^2(I) X_i+\sigma^2\mu_i(I)}{\tau_i^2(I)+\sigma^2}, 
		\frac{\tau_i^2(I)\sigma^2}{\tau_i^2(I)+\sigma^2}\Big).
	\end{align}

	Next, introduce the prior $\lambda$ on $\mathcal{I}$,
	discussed in Section \ref{subsec_duscussion} below.
	For $\varkappa>1$, draw a random set  from $\mathcal{I}$ with probabilities 
	\begin{align}
		\label{prior_lambda}
		\lambda_I=c_{\varkappa,n} \exp\big\{-\varkappa|I| \log (\tfrac{en}{|I|})\big\} 
		= c_{\varkappa,n}(\tfrac{en}{|I|})^{-\varkappa |I|}, \quad I \in\mathcal{I},
	\end{align} 
	where $c_{\varkappa,n}$ is the normalizing constant. 
	Since $(\frac{n}{k})^k \le \binom{n}{k} \le (\frac{en}{k})^k$ and $\binom{n}{0}=1$,  
	\begin{align}
		\label{sum_lambda}
		1&=\sum_{I \in\mathcal{I}}\lambda_I=c_{\varkappa,n}\sum_{k=0}^n \binom{n}{k} 
		\Big(\frac{en}{k}\Big)^{-\varkappa k}\le
		c_{\varkappa,n}\sum_{k=0}^n \Big(\frac{en}{k}\Big)^{-(\varkappa-1)k}
		\le c_{\varkappa,n}\sum_{k=0}^n e^{-(\varkappa-1)k},
	\end{align} 
	so that $c_{\varkappa,n}\ge 1-e^{1-\varkappa}>0$, $n \in \mathbb{N}$. 
	Combining \eqref{gen_norm_prior} and \eqref{prior_lambda} gives 
	the mixture prior on $\theta$: $\pi = \sum_{I \in\mathcal{I}} \lambda_I \pi_I$.
	This leads to the marginal distribution of $X$: 
	$\mathrm{P}_X=\sum_{I \in\mathcal{I}} 
	\lambda_I \mathrm{P}_{X,I}$, with  
	$\mathrm{P}_{X,I}= \bigotimes_{i=1}^n
	\mathrm{N}\big(\mu_i(I), \sigma^2+\tau_i^2(I)\big)$,
	and the posterior of $\theta$ is
	\begin{align}
		\label{norm_gen_posterior}
		\pi(\vartheta|X) =\pi_\varkappa(\vartheta|X) =\sum_{I \in\mathcal{I}} \pi_I(\vartheta|X)\pi(I|X),
	\end{align}
	where   
	$\pi_I(\vartheta|X)$ is defined by (\ref{norm_poster_I}) 
	and the posterior $\pi(I|X)$  for $I$ is
	\begin{align}
		\label{norm_P(I|X)}
		\pi(I|X)=
		\frac{\lambda_I \mathrm{P}_{X,I}}{\mathrm{P}_X}=
		\frac{\lambda_I \prod_{i=1}^n 
			\phi\big(X_i,\mu_i(I), \sigma^2+\tau_i^2(I)\big)}
		{\sum_{J\in\mathcal{I}}\lambda_J\prod_{i=1}^n 
			\phi\big(X_i,\mu_i(J), \sigma^2+\tau_i^2(J)\big)}.
	\end{align}
	
	\subsection{Empirical Bayes posterior} 
	The parameters $\mu_{1,i}$ are yet to be chosen in the prior. We choose $\mu_i$
	by using empirical Bayes approach. The marginal likelihood $\mathrm{P}_X$ is 
	readily maximized with respect to $\mu_i$: $\tilde{\mu}_i=X_i$, which we then 
	substitute instead of $\mu_i$ in the expression (\ref{norm_gen_posterior}) 
	for $\pi(\vartheta|X)$, obtaining the empirical Bayes posterior
	\begin{align}
		\label{emp_norm_posterior}
		\tilde{\pi}(\vartheta|X) = \tilde{\pi}_\varkappa(\vartheta|X)=
		\sum_{I \in\mathcal{I}} \tilde{\pi}_I(\vartheta|X)\tilde{\pi}(I|X),
	\end{align}
	where the empirical Bayes conditional posterior 
	(recall that $\mathrm{N}(0,0) =\delta_0$)
	\begin{align}
		\tilde{\pi}_I(\vartheta|X)&=
		\label{emp_poster_I}
		\bigotimes_{i=1}^n \mathrm{N}\big(X_i \mathbbm{1}\{ i\in I\}, 
		\tfrac{K_n(I)\sigma^2\mathbbm{1}\{ i\in I\}}{K_n(I)+1} \big)
	\end{align}
	is obtained from \eqref{norm_poster_I} 
	with $\mu_{1,i}=X_i$, and 
	\begin{align}
		\label{emp_P(I|X)}
		\tilde{\pi}(I|X)=
		\frac{\lambda_I \mathrm{P}_{X,I}}
		{\sum_{J\in\mathcal{I}} \lambda_J\mathrm{P}_{X,J}}
		=\frac{\lambda_I\prod_{i=1}^n
			\phi (X_i,X_i\mathbbm{1}\{i\in I\},\sigma^2+\tau^2_i(I))}
		{\sum_{J\in\mathcal{I}} \lambda_J\prod_{i=1}^n 
			\phi(X_i, X_i\mathbbm{1}\{i\in J\},\sigma^2+\tau^2_i(J))}
	\end{align} 
	is the empirical Bayes posterior for $I \in\mathcal{I}$,
	obtained from \eqref{norm_P(I|X)}
	with $\mu_i=X_i$.
	Let $\tilde{\mathrm{E}}$ and $\tilde{\mathrm{E}}_I$ be the expectations 
	with respect to the measures $\tilde{\pi}(\vartheta|X)$ and $\tilde{\pi}_I(\vartheta|X)$ 
	respectively. Then $\tilde{\mathrm{E}}_I(\vartheta|X)=X(I)=(X_i\mathbbm{1}\{i\in I\}, i\in[n])$.
	Introduce the \emph{empirical Bayes posterior mean} estimator
	\begin{align}
		\label{estimator}
		\tilde{\theta}&=\tilde{\mathrm{E}} (\vartheta|X) =
		\sum_{I\in\mathcal{I}}\tilde{\mathrm{E}}_I(\vartheta|X) \tilde{\pi}(I|X)
		=\sum_{I\in\mathcal{I}}X(I)\tilde{\pi}(I|X).
	\end{align}

	Consider an alternative empirical Bayes posterior. First derive an empirical 
	Bayes variable selector  $\hat{I}$ by maximizing $\tilde{\pi}(I|X)$ over $I \in\mathcal{I}$
	(any maximizer will do) as follows:
	\begin{align}
		\hat{I}\!&=\!\arg\!\max_{I\in\mathcal{I}}{\tilde{\pi}}(I|X)
		=\arg\!\max_{I\in\mathcal{I}}\lambda_I \mathrm{P}_{X,I}
		=\arg\!\max_{I\in\mathcal{I}} \Big\{\!-\!\sum_{i \in I^c } \tfrac{X_i^2}{2\sigma^2}
		-\tfrac{|I|}{2}\log (K_n(I)+1)+\log \lambda_I\Big\} \notag\\
		\label{I_MAP}
		&=
		\arg\!\min_{I\in\mathcal{I}}\Big\{\sum_{i \in I^c}X_i^2 +
		(2\varkappa +1) \sigma^2 |I| \log\big(\tfrac{en}{|I|}\big)\Big\},
	\end{align}
	which is reminiscent of the penalization procedure from \cite{Birge&Massart:2001} 
	(cf.\ also \cite{Abramovich&Grinshtein&Pensky:2007}).
	Now plugging in $\hat{I}$ into $\tilde{\pi}_I(\vartheta|X)$ defined by \eqref{emp_poster_I} 
	yields another empirical (now 
	with respect to $\mu_i$'s and $I$)  Bayes posterior
	and the corresponding empirical Bayes mean estimator for $\theta$:
	\begin{align}
		\label{emp_emp_posterior}
		\check{\pi} (\vartheta|X)=
		\tilde{\pi}_{\hat{I}} (\vartheta|X), \quad  
		\check{\theta}= \check{\mathrm{E}}(\vartheta|X)= X(\hat{I})=(X_i\mathbbm{1}\{i\in \hat{I}\}, i\in[n]),
	\end{align}
	where $\check{\mathrm{E}}$ denotes the expectation with respect to the 
	measure $\check{\pi}(\vartheta|X)$. 
	
	\subsection{Conditions} 
	We provide some technical conditions and definitions. 
	The following condition (called \emph{exchangeable exponential moment condition})
	on the error vector $\xi=(\xi_1,\ldots,\xi_n)$ will be assumed throughout.  \smallskip\\
	{\sc Condition (A1).}
	The random variables $\xi_i$'s from \eqref{model} satisfy: 
	$\mathrm{E} \xi_i =0$, $\mathrm{Var}(\xi_i)\le C_\xi$, $ i\in[n]$; and 
	for some $\beta, B>0$ (without loss of generality assume $C_\xi =1$ and 
	$\beta\in (0,1]$),
	\begin{align}
		\label{cond_nonnormal}
		\tag{A1}
		\mathrm{E} \exp\Big\{\textstyle{\beta\sum_{i\in I}} \xi_i^2\Big\} \le e^{B |I|} \quad
		\text{for all} \;\; I\in  \mathcal{I}.
	\end{align}  
	Notice that the unknown distribution of $\xi$ may also depend on $\theta$, in that case 
	we assume Condition (A1) to be fulfilled for all $\theta\in\mathbb{R}^n$.
	The constants $\beta \in (0,1]$  and $B>0$ will be fixed in the sequel and we 
	omit  the dependence on these constants in all further notation. 
	A short discussion about this condition can be found in Section \ref{subsec_duscussion}.
	There is no need to assume 
	$\mathrm{Var}(\xi_i)\le C_\xi$ as this follows from \eqref{cond_nonnormal}, 
	but we provide this just for reader's convenience.
	Note also that the $\xi_i$'s do not have to be even independent. For example, $\xi_i$'s 
	can follow an autoregressive model, see Section \ref{subsec_duscussion}.
	
	Condition \eqref{cond_nonnormal} is clearly satisfied for independent normals 
	$\xi_i \overset{\rm ind}{\sim} \mathrm{N}(0,1)$ and for bounded (arbitrarily dependent) $\xi_i$'s.
	In case of independent normal errors, some bounds in the proofs can be sharpened;
	we will mention possible refinements  in Section \ref{subsec_duscussion}.

	In the proof of Theorem \ref{th1} below, we will need a bound for 
	$\mathrm{E}\big(\sum_{i\in I} \xi_i^2 \big)^2$, $I\in\mathcal{I}$.  
	Actually, Condition \eqref{cond_nonnormal}  ensures such a bound. 
	Indeed, since $x^2 \le e^{2x}$ for all $x\ge 0$, by using the H\"older inequality and  \eqref{cond_nonnormal}, we derive that for any $t\in (0,\beta]$ 
	\[
	\mathrm{E}_{\theta}\big(\sum_{i\in I} \xi_i^2 \big)^2
	=\frac{4}{t^2}  \mathrm{E}_{\theta}\big(\tfrac{t}{2}\sum_{i\in I} \xi_i^2 \big)^2 \le 
	\frac{4}{t^2}  \mathrm{E}_{\theta}e^{t\sum_{i\in I} \xi_i^2}
	\le \frac{4}{t^2}\big[\mathrm{E} e^{\beta\sum_{i\in I} \xi_i^2}\big]^{t/\beta}\le \frac{4}{t^2}e^{Bt\beta^{-1} |I|}.
	\]
	To summarize, Condition \eqref{cond_nonnormal} implies that for any 
	$\rho\in(0,B/2]$ and any $I\in  \mathcal{I}$, 
	\begin{align}
		\label{moment_bound}
		\mathrm{E}\big(\textstyle{\sum_{i\in I}} \xi_i^2 \big)^2 
		\le \frac{B^2}{(\beta\rho)^2}e^{2\rho |I|}.
	\end{align}

	In some proofs we need a technical condition on the parameter $\varkappa$ 
	appearing in \eqref{prior_lambda}.\smallskip\\
	{\sc Condition (A2).} The parameter $\varkappa$ of the prior 
	$\lambda$ defined by \eqref{prior_lambda} satisfies
	\begin{align}
		\label{cond_technical}
		\tag{A2}
		\varkappa>\bar{\varkappa}\triangleq(12-\beta+4B)/(4\beta),
	\end{align}  
	where $\beta, B$ are from Condition \eqref{cond_nonnormal}.
	
	Finally, we give one more technical definition which we will need in the claims.
	For constants $\beta,B$ from Condition \eqref{cond_nonnormal}, define 
	\begin{align}
		\label{def_bar_tau}
		\bar{\tau}=\bar{\tau}(\varkappa,\beta,B)\triangleq 
		3(\varkappa\beta+\tfrac{\beta}{2}+B)/\beta.
	\end{align}

	\section{Main results}
	\label{main_results}
	In this section we give the main results of the paper.
	From now on, by $\hat{\pi}(\vartheta|X)$ (with corresponding expectation $\hat{\mathrm{E}}(\cdot |X)$)
	we denote either $\tilde{\pi}(\vartheta|X)$ defined by \eqref{emp_norm_posterior} 
	or $\check{\pi}(\vartheta|X)$ defined by \eqref{emp_emp_posterior},
	and $\hat{\theta}$ will stand either for $\tilde{\theta}$  defined by 
	\eqref{estimator} or for $\check{\theta}$ defined by \eqref{emp_emp_posterior}.
	Also, $\hat{\pi}(I \in \mathcal{G}|X)$ should be read as 
	$\tilde{\pi}(I \in \mathcal{G}|X)$ in case $\hat{\pi}=\tilde{\pi}$, 
	and as $\mathbbm{1}\{\hat{I}\in\mathcal{G}\}$ in case $\hat{\pi}=\check{\pi}$, 
	for all $\mathcal{G}\subseteq \mathcal{I}$ that appear in what follows.
	Hence, $\hat{\pi}(I|X) =\tilde{\pi}(I|X)$ 
	and $\mathrm{E}_{\theta}\hat{\pi}(I\in\mathcal{G}|X)
	=\mathrm{E}_{\theta}\tilde{\pi}(I\in\mathcal{G}|X)$ in the former case,
	and  $\hat{\pi}(I|X)=\mathbbm{1}\{\hat{I}=I\}$ and $\mathrm{E}_{\theta}\hat{\pi}(I\in\mathcal{G}|X)
	=\mathrm{P}_{\theta}(\hat{I}\in\mathcal{G})$ in the latter case.
	
	\subsection{Oracle rate}
	\label{sec_oracle_rate}
	The empirical Bayes posterior $\hat{\pi}(\vartheta|X)$ is a random mixture over 
	$\tilde{\pi}_I(\vartheta|X)$, $I\in\mathcal{I}$. 
	From the $\mathrm{P}_{\theta}$-perspective, each posterior 
	$\tilde{\pi}_I(\vartheta|X)$ 
	(and the corresponding estimator $\tilde{\mathrm{E}}_I(\vartheta|X)=X(I)$)
	contracts to the true parameter $\theta$ with the local rate
	$R^2(I,\theta)=\sum_{i\in I^c} \theta_{i}^2+\sigma^2|I|$.
	Indeed, since $\tilde{\mathrm{E}}_I(\vartheta|X)=X(I)=(X_i\mathbbm{1}\{i\in I\}, i\in\mathbb{N}_n)$,
	\eqref{emp_poster_I} and the Markov inequality yields
	\begin{align*}
		\mathrm{E}_{\theta}\tilde{\pi}_I( \|\vartheta-\theta\|^2&\ge M^2 R^2(I,\theta)| X)
		\le \frac{\mathrm{E}_{\theta}
			\|X(I)-\theta\|^2+\frac{K_n(I)\sigma^2|I|}{K_n(I)+1}}{ M^2 R^2(I,\theta)}
		\le \frac{2}{M^2}.
	\end{align*}
	For each $\theta\in\mathbb{R}^n$, among $I \in \mathcal{I}$ there exists the 
	best choice $I_o=I_o(\theta)=I_o(\theta,\sigma)$ (called the \emph{$R$-oracle}) 
	corresponding to the fastest local rate 
	$
	R^2(\theta)=R^2(\theta,\mathcal{I})=\min_{I\in\mathcal{I}}R^2(I,\theta)
	=\sum_{i\in  I_o^c} \theta_{i}^2+\sigma^2|I_o|.
	$
	Ideally, we would like to \emph{mimic} the $R$-oracle, i.e., to construct
	an empirical Bayesian procedure (e.g., $\hat{\pi}(\vartheta|X)$) which performs as good as 
	the oracle empirical Bayes posterior $\tilde{\pi}_{I_o}(\vartheta|X)$ without knowing $I_o$, 
	uniformly in $\theta\in \mathbb{R}^n$. 
	However, the lower bounds for the estimation problem (hence, also for the posterior contraction problem),
	obtained by \cite{Donoho&Johnstone:1994a} and later by \cite{Birge&Massart:2001},
	show that it is impossible to mimic the $R$-oracle and a logarithmic factor is the unavoidable price for
	the uniformity over $\mathbb{R}^n$ (otherwise this would contradict to the minimax lower bound over the scale of sparsity classes, cf.\ \cite{Birge&Massart:2001}). 
	Therefore only a modification of the risk $R$-oracle where the variance
	term $\sigma^2|I_o|$ is inflated with the factor $\log(en/|I_o|)$ (thought of as payment for not knowing $I_o$) 
	is ``mimicable".
	

	The above discussion motivates the following definition.
	Introduce the  family of
	local rates
	\begin{align}
		\label{oracle_I}
		r^2(I,\theta)=r^2_\sigma(I,\theta)=\sum_{i\in I^c} \theta_{i}^2+\sigma^2|I| \log(\tfrac{en}{|I|})
		=B(I,\theta) +V(I), \quad I\in\mathcal{I}, 
	\end{align}
	where $B(I,\theta) = \sum_{i\in I^c} \theta_{i}^2$  is  the bias part of the rate
	and $V(I)=V(I,\sigma,n)=\sigma^2|I| \log(\tfrac{en}{|I|})$ is the adjusted variance part, 
	the variance term $\sigma^2|I|$ of the rate $R(I,\theta)$ multiplied by 
	the logarithmic factor $\log(\tfrac{en}{|I|})$.
	There exists the best choice $I_o=I_o(\theta)=I_o(\theta,\sigma^2)=I_o(\theta,\sigma^2,n)
	\in\mathcal{I}$ (called \emph{oracle}) at which the rate \eqref{oracle_I} is minimal: 
	\begin{align}
		\label{oracle}
		r^2(\theta)&=r^2(\theta,\mathcal{I})=\min_{I\in\mathcal{I}}  r^2(I,\theta)
		=r^2(I_o(\theta),\theta)
		=B(I_o(\theta),\theta)+V(I_o(\theta)),
	\end{align}
	called the \emph{oracle rate}. 
	Note that the oracle $I_o$ may not be unique (but $|I_o|$ is unique) 
	as some coordinates of $\theta$ can coincide, 
	in that case take the one with the earliest coordinates.
	Clearly,
	$I_o= \{i\in[n]: \theta_i^2\ge \theta_{[i_o]}^2\}$, where $i_o=|I_o|=\arg\!\min_{k\in[n]_0}\{
	\sum_{i=1}^{n-k}\theta_{(i)}^2 + \sigma^2 k \log(\tfrac{en}{k}\}$.
	Thus the oracle $I_o$ classifies
	the coordinates $(\theta_i, i \in I_o)$ as \emph{significant}
	and  the coordinates  $(\theta_i, i \in I_o^c)$ as \emph{insignificant}.
	The bias related term $B(I_o(\theta),\theta)=\sum_{i\in I_o^c} \theta_{i}^2
	=\sum_{i=1}^{n-i_o}\theta_{(i)}^2 $ of the oracle rate  is called 
	the \emph{excessive bias}. This is the error the oracle makes 
	when setting insignificant coordinates of $\theta$ to zero.
	The variance related term $\sigma^2 |I_o|\log(\tfrac{en}{|I_o|})$
	is the error the oracle makes when recovering the significant coordinates
	(the log factor is the payment for not knowing the locations). 
	The definition \eqref{oracle} of the oracle $I_o$ implies the following characterization 
	of the significant coordinates $\{\theta_{[i]}, i=1,\ldots, i_o\}$:
	\begin{equation}
		\label{chacter_signif}
		\begin{aligned}
			\theta^2_{[i_o]} &\ge \sigma^2\big[\log(\tfrac{en}{i_o})-(i_o-1)\log(\tfrac{i_o}{i_o-1})\big],\\
			\theta^2_{[i_o]}+\theta^2_{[i_o+1]}&\ge 
			\sigma^2\big[2\log(\tfrac{en}{i_o})-(i_o-2)\log(\tfrac{i_o}{i_o-2})\big],\\
			&\ldots, \\
			\textstyle{\sum_{i=1}^{i_o}}\theta_{[i]}^2 &\ge\sigma^2 i_o \log(\tfrac{en}{i_o}).
		\end{aligned}
	\end{equation}
	The insignificant coordinates $\{\theta_{(i)}, i=1,\ldots, n-i_o\}$ can be characterized in a similar manner.
	

	Introduce a family of the so called $\tau$-oracles $I_o^\tau=I_o^\tau(\theta)=I_o(\theta, \tau\sigma^2)$, 
	$\tau\ge 0$ and let $i_\tau=|I_o^\tau(\theta)|$ be the corresponding cardinalities. 
	A $\tau$-oracle $I_o^\tau(\theta)$ is just the usual oracle defined by \eqref{oracle} 
	with  $\sigma^2$ substituted by $\tau\sigma^2$, the oracle itself is the $\tau$-oracle
	with $\tau=1$: $I_o(\theta)=I_o^1(\theta)$.  Notice that $I_o^{\tau_1}\subseteq I_o^{\tau_2}$ 
	if $\tau_1\ge\tau_2$. 
	For $\tau \downarrow 0$, $r^2(\theta, I_o^\tau)\downarrow 0$  and
	the ``limiting'' $\tau$-oracle recovers the active index set 
	$I^*=I^*(\theta)=\{i\in [n]:\, \theta_{i} \not =0\}$ in the sense that  
	$I_o^\tau \uparrow I^*$ as $\tau \downarrow 0$.
	Informally, since the $\tau$-oracle is defined by substituting $\tau\sigma^2$ instead of 
	$\sigma^2$ in the oracle rate, one can think of the $\tau$-oracle with $\tau\in[0,1)$ as if $X$ 
	is observed with a ``magnifying glass'' 
	since the error variance is reduced by the factor $\tau$ so that the $\tau$-oracle 
	can distinguish more coordinates from zero. 
	In the case $\tau>1$, the error variance in \eqref{model} increases by the factor $\tau$ 
	(as if the observations $X_i$'s get blurred), resulting in a smaller set of significant 
	coordinates recovered by the $\tau$-oracle.
	However, all $\tau$-oracle rates $r^2(\theta, I_o^\tau)$, $\tau>0$, are related to 
	the oracle rate $r^2(\theta)=r^2(\theta, I_o^1)$ by the trivial relations
	\[
	r^2(\theta)\le r^2(\theta, I_o^\tau)\le \tau r^2(\theta) \quad \text{for} \;\; \tau\ge 1, \quad 
	r^2(\theta, I_o^\tau) \le r^2(\theta) \le \tau^{-1} r^2(\theta,I_o^\tau) \quad \text{for} \;\; 0<\tau< 1.
	\]
	So, in principle we can obtain the result for any $\tau$-oracle rate $r^2(\theta,I_o^\tau)$
	via the result for the oracle rate $r^2(\theta)$ and vice versa, but  
	at the price of some multiplicative constant. 
	
	Actually, we can look at all $\tau$-oracles $I_o^\tau$, $\tau\ge 0$,
	from the following general perspective. Introduce a family of $n+1$ sets
	\begin{align}
		\label{family_I_o}
		\mathcal{I}_o= \{I_o(k), \, k\in[n]_0\}, \quad \text{where} 
		\quad I_o(k) = \{i\in[n]: \theta^2_i\ge \theta_{[k]}^2\}.
	\end{align} 
	Clearly, these are embedded sets $\varnothing\triangleq I_o(0) 
	\subseteq I_o(1)\subseteq I_o(2) \subseteq \ldots \subseteq I_o(n)=[n]$.
	Now notice that the oracle set 
	$I_o(\theta)$ and actually all $\tau$-oracles $I_o^\tau(\theta)$, $\tau\ge 0$,
	are all from this family, in fact, $I_o = I_o(i_o)$ and $I_o^\tau=I_o(i_\tau)$.

	%

	\subsection{Contraction results with oracle rate}
	The following theorem establishes that the empirical Bayes posterior $\hat{\pi}(\vartheta|X)$ 
	(which is either $\tilde{\pi}(\vartheta|X)$ defined by \eqref{emp_norm_posterior} 
	or $\check{\pi}(\vartheta|X)$ defined by \eqref{emp_emp_posterior})
	contracts to $\theta$  with the oracle rate $r(\theta)$ from the frequentist  
	$\mathrm{P}_{\theta}$-perspective,  and the empirical Bayes posterior mean $\hat{\theta}$
	which is either $\tilde{\theta}$  defined by \eqref{estimator} or $\check{\theta}$ 
	defined by \eqref{emp_emp_posterior}) converges to $\theta$ with the oracle rate $r(\theta)$, uniformly 
	over the entire parameter space.

	\begin{theorem} 
		\label{th1} 
		Let Conditions \eqref{cond_nonnormal} and \eqref{cond_technical}
		be fulfilled. Then there exist positive constants  $M_0, M_1, H_0, H_1, m_0, m_1$ 
		such that for any $\theta\in\mathbb{R}^n$ and any $M\ge 0$,
		\begin{align}
			\tag{i}
			\label{th1_i}
			\mathrm{E}_{\theta}\hat{\pi}\big( \|\vartheta-\theta\|^2
			\ge M_0r^2(\theta) +M\sigma^2 | X\big)  &\le H_0 e^{-m_0 M},\\
			\tag{ii}
			\label{th1_ii}
			\mathrm{P}_{\theta}\big( \|\hat{\theta}-\theta\|^2\ge M_1 r^2(\theta)+M\sigma^2\big)
			&\le H_1 e^{-m_1 M}.
		\end{align}
	\end{theorem}
	\begin{remark}
		\label{rem2_or_in}
		Notice that already claim \eqref{th1_i} of the theorem contains an oracle bound 
		for the estimator $\hat{\theta}$. Indeed, by Jensen's inequality, we derive the oracle inequality 
		\begin{align}
			\label{cor_est}
			\mathrm{E}_{\theta}\|\hat{\theta}-\theta\|^2 \le
			\mathrm{E}_{\theta} \hat{\mathrm{E}}(\|\vartheta-\theta\|^2|X)
			\le M_0r^2(\theta) + H_0 \int_0^{+\infty} \!\!\!\! e^{-m_0u/\sigma^2} du
			=M_0r^2(\theta)  + \frac{H_0\sigma^2}{m_0}. 
		\end{align}
		However, comparing claim \eqref{th1_ii} with the oracle inequality \eqref{cor_est}, 
		we see that claim \eqref{th1_ii} is actually stronger (and more refined) than 
		\eqref{cor_est} and therefore requires a separate proof.
	\end{remark}
	
	\begin{remark}
		A few more remarks on the theorem are in order.
		\begin{itemize}
			\item[\rm{(i)}]
			The above local result implies  the minimax 
			optimality over various sparsity scales, see Section \ref{implication} for more detail on this.
			\item[\rm{(ii)}]
			The constants  $M_0, M_1, H_0, H_1, m_0, m_1>0$ in the theorem depend only on $\beta, B$ 
			and some also on $\varkappa$, the exact expressions can be found in the proof.
			\item[\rm{(iii)}]
			The non-asymptotic exponential bounds in terms of the constant $M$
			from the expression $M'r^2(\theta) +M\sigma^2$ (with some fixed $M'$)  in claims (i) and (ii)
			of the theorem provide a very refined characterization of the quality of the posterior 
			$\hat{\pi}(\vartheta|X)$ and estimator $\hat{\theta}$, finer 
			than, e.g., the traditional oracle inequalities like \eqref{cor_est}.
			This refined formulation allows for subtle analysis in various asymptotic regimes ($n \to \infty$, $\sigma \to 0$, 
			or their combination) 
			as we can let $M$ depend in any way on $n$, $\sigma$, or both. 
		\end{itemize}
	\end{remark}

	The next theorem describes the frequentist behavior of 
	the selector $\hat{I}$ and the empirical Bayes posterior for $I$, saying basically 
	that $\hat{I}$ and $\tilde{\pi}(I|X)$ ``live'' on a certain set that is, in a sense, almost as good as the oracle $I_o=I_o(\theta)$ defined by \eqref{oracle}. 
	For any $\theta\in\mathbb{R}^n$, introduce
	\begin{align}
		\label{def_I_*}
		I_*=I_*(\theta) \triangleq I_o^{\tau_0}(\theta)=I_o(\theta,\tau_0\sigma^2), \quad i_*=|I_*|,
	\end{align}
	where we fix some $\varrho\in(0,1)$ and  $\tau_0> \tfrac{1+\varrho}{1-\varrho} \bar{\tau}$,
	$\bar{\tau}$ is defined by \eqref{def_bar_tau}. For example, we can take $\varrho = 0.1$ and 
	$\tau_0=\tfrac{11}{9}\bar{\tau}+0.1$.  
	\begin{theorem} 
		\label{th2} 
		Let Condition \eqref{cond_nonnormal} be fulfilled. The following relations hold
		for any $\theta\in\mathbb{R}^n$, $M\ge 0$.
		\begin{itemize}
			\item[(i)] Let $\varkappa>
			\tfrac{4+\beta+2B}{2\beta}$ (Condition \eqref{cond_technical} implies this).
			There exist $M'_0, H'_0>0$ such that 
			\begin{align*}
				&\mathrm{E}_{\theta}\hat{\pi}\big(I\in\mathcal{I}: |I| \log(\tfrac{en}{|I|})\ge  
				M'_0 |I\cap I_o| \log(\tfrac{en}{|I\cap I_o|})+M \big| X\big)\le H'_0 
				e^{-M}.
			\end{align*}
			\item[(ii)]
			Let  $\varkappa>\beta^{-1}-\tfrac{1}{2}$ 
			(Condition \eqref{cond_technical} implies this),
			$\bar{\tau}$ be defined by \eqref{def_bar_tau}. Fix any $I'\in\mathcal{I}$.
			Then there exist $H'_1, m'_0>0$ (independent of $\theta$ and $I'$) such that 
			\begin{align*}
				&\mathrm{E}_{\theta}\hat{\pi}\big(I\in\mathcal{I}:\, \sum_{i\in I'\backslash I} \tfrac{\theta^2_i}{\sigma^2} 
				\ge  \bar{\tau}|I\cup I'| \log(\tfrac{en}{|I\cup I'|})+M |X\big) \le 
				H'_1e^{-m'_0M}.
			\end{align*}
			In particular,  let $I_*$ be defined by \eqref{def_I_*},
			then there exist $\alpha'_1,m'_1> 0$ such that 
			\begin{align}
				\label{th2_ii}
				&\mathrm{E}_{\theta}\hat{\pi}\big(I\in\mathcal{I}: |I| \log(\tfrac{en}{|I|})\le 
				\varrho |I_*| \log(\tfrac{en}{|I_*|}) - M \big|X\big) \le 
				H'_1\big(\tfrac{en}{ |I_*|}\big)^{-\alpha'_1|I_*|} e^{-m'_1M}.
			\end{align}
			
			\item[(iii)]
			Let Condition \eqref{cond_technical} be fulfilled,
			$c_1, c_2, c_3$ be the constants defined in Lemma \ref{lemma2}. Then
			\begin{align*}
				\mathrm{E}_{\theta}\hat{\pi}(I \in \mathcal{I}: r^2(I,\theta)\ge c_3 r^2(\theta) + M \sigma^2|X) 
				\le  C_0 e^{-c_2 M}, \quad \text{where}\;\;
				C_0 =(1-e^{1-c_1})^{-1} .
			\end{align*}
		\end{itemize}
	\end{theorem}
	
	\begin{remark} 
		The assertion \eqref{th2_ii} holds also for $I_*$  defined differently:
		\begin{align}
			\label{def_I*}
			I_*=I_*(\theta)=I_*(\theta,\tau)=I_*(\theta, \tau,\varrho)= \{i\in[n]: \theta_i^2 \ge\theta_{[i_*]}^2 \}
		\end{align}
		and $i_*=i_*(\tau,\varrho,\theta)=\max\{k\in[n]_0: 
		\sum_{\varrho k}^k\theta_{[i]}^2\ge\tau(1-\varrho)\sigma^2 k \log(\tfrac{en}{k})\}$.
		Indeed, the only difference in the proof of  \eqref{th2_ii} for $I_*$ defined by \eqref{def_I*} 
		is that, instead of \eqref{proof_th2_ii}, we have the bound
		\[
		\textstyle{\sum_{i\in I_*\backslash I}} \tfrac{\theta^2_i}{\sigma^2} 
		\ge\textstyle{\sum_{\varrho |I_*|}^{|I_*|}}\tfrac{\theta_{[i]}^2}{\sigma^2} 
		\ge \tau (1-\varrho) |I_*| \log(\tfrac{en}{|I_*|}) \ge 
		\tau' |I\cup I_*|\log(\tfrac{en}{|I \cup I_*|})
		+\tau' M,
		\]
		so that $m'_1=\tau' m'_0$ in this case and the rest of the proof is the same.
	\end{remark}
	
	For any $\theta\in\mathbb{R}^n$, the set $I_*(\theta)$ is the representative 
	from the family $\mathcal{I}_o$ (defined by 
	\eqref{family_I_o}) that consists of
	``distinctly significant'' coordinates of $\theta$ such that 
	$\hat{\pi}(I|X)$ makes (almost) no mistake for selecting a 
	big proportion of this set. 
	Recall that for $\tau_1\ge \tau_2$, $I_o(i_{\tau_1})=I_o^{\tau_1}\subseteq I_o^{\tau_2}= I_o(i_{\tau_2})$ 
	so that $i_{\tau_1} \le i_{\tau_2}$.
	Since $I_*=I_o^{\tau_0}$ for $\tau_0>1$, we have $I_*\subseteq I_o$. 
	So, the claims of the above theorem roughly mean that $\hat{\pi}(I|X)$ 
	(i.e., the selector $\hat{I}$ and the posterior $\tilde{\pi}(I|X)$) 
	lives in the ``shell'' $\{I: I_o(K^{-1} i_*) \subseteq I \subseteq I_o(Ki_o)\}$ 
	for some sufficiently large $K$, where the sets $I_o(k)$ are defined by \eqref{family_I_o}. 
	So, if $I_*$ and $I_o$ are ``close'' to each other (i.e., $i_o \le C i_*$ for some $C>0$), 
	then $\hat{\pi}(I|X)$ recovers well the oracle structure $I_o$. 
	The case $i_* \ll i_o$ is problematic (the corresponding 
	$\theta$ is ``deceptive'') as the living shell for $\hat{\pi}(I|X)$ is then too wide.
	Property (i) of Theorem \ref{th2} claims good ``over-dimensionality'' control of $\hat{\pi}(I|X)$ 
	in terms of the oracle $I_o$. In other words, the method does a good job in assigning 
	insignificant coordinates to zeros, for any $\theta\in\mathbb{R}^n$. On the other hand, 
	there is no full ``under-dimensionality''  control for $\hat{\pi}(I|X)$, as property (ii) is only in terms 
	of the set $I_*$ (which may be much ``smaller'' than $I_o$): 
	basically for deceptive $\theta$'s, $\hat{\pi}(I|X)$ can make relatively many errors by assigning many 
	significantly non-zero coordinates to zeros.
	
	This is reminiscent of the same asymmetric situation for adaption to smoothness 
	where it is also possible to control under-smoothing (e.g., by penalization procedures or 
	Lepski's method), but not over-smoothing. 
	In view of the lower bound results mentioned in the introduction, this is not an artefact of 
	the method, it is a fundamental, unavoidable problem. It occurs for the so called deceptive
	parameters $\theta$ that have many smallish coordinates, just slightly under the noise level.  
	Interestingly, controlling over-dimensionality 
	(or under-smoothing for smoothness structures) is enough for solving adaptive estimation problem, 
	but not for uncertainty quantification where we need both over-dimensionality control 
	(for the optimal size) and under-dimensionality control (for the coverage).
	This is possible for the non-deceptive parameters described by the so called EBR condition 
	and introduced in the next section.
	

	

	\subsection{Confidence ball under excessive bias restriction}
	\label{subsec_conf_set}
	
	Theorem \ref{th1} establishes the strong local optimal properties 
	of the empirical Bayes posterior $\hat{\pi}(\vartheta|X)$
	and the empirical Bayes posterior mean $\hat{\theta}$, 
	but these do not solve the uncertainty quantification problem yet.
	Let us construct a confidence ball by using the empirical Bayes 
	posterior $\check{\pi}(\vartheta|X)$ defined by \eqref{emp_emp_posterior}.
	Since $\check{\pi}(\vartheta|X)=\bigotimes_{i=1}^n\mathrm{N}(\check{\theta}_i, \check{\sigma}^2_i)$
	with $\check{\theta}_i=X_i \mathbbm{1}\{i\in\hat{I}\}$ and
	$\check{\sigma}^2_i= (1-|\hat{I}|/en)\sigma^2 \mathbbm{1}\{i\in \hat{I}\}$, 
	denoting by $\chi^2_{k,\alpha}$ the $(1-\alpha)$-quantile of $\chi^2_k$-distribution we have 
	\[
	\check{\pi}\big(\|\vartheta-\check{\theta}\|^2 \le\sigma^2 \chi^2_{|\hat{I}|,\alpha} |X\big)\ge
	\check{\pi}\big(\|\vartheta-\check{\theta}\|^2\le(1-|\hat{I}|/en)\sigma^2\chi^2_{|\hat{I}|,\alpha}
	|X\big)=1-\alpha.
	\]
	But $\chi^2_{|\hat{I}|,\alpha}$ is bounded by a constant multiple of $|\hat{I}|$, and 
	hence for simplicity the latter can replace the former to obtain a credible ball.
	This leads to $B(\check{\theta},M\sigma {|\hat{I}|^{1/2}})$ as a credible ball for $\theta$, which 
	can be guaranteed to have at least a given level of credibility by choosing 
	a sufficiently large constant $M$. 
	From (i) of Theorem \ref{th2} it follows that 
	$|\hat{I}|$ is of the order $|I_o|$. 
	However, it is clear that $B(\check{\theta},M\sigma |\hat{I}|^{1/2})$ 
	cannot have a guaranteed coverage, since otherwise the center $\check{\theta}$ 
	would be an estimator that mimics the $R$-oracle uniformly in $\theta\in\mathbb{R}^n$, which is impossible as we discussed earlier.
	Hence to obtain coverage, the order of the radius of any confidence ball must contain a logarithmic factor. This leads us to the inflated credible ball $B(\check{\theta}, M\hat{r})$, where 
	\begin{align}
		\label{check_r}
		\hat{r}^2=\hat{r}^2(X)= \sigma^2+\sigma^2 |\hat{I}|\log(en/|\hat{I}|).
	\end{align}

	The empirical Bayes posterior $\check{\pi}(\vartheta|X)$ is well concentrated 
	(in fact, in a ball of the size $M\sigma^2|I_o|$), but not around the truth, rather around
	its mean $\check{\theta}$ which in general is away from the truth 
	by the distance at most of the order of the oracle rate $r(\theta)$. 
	We can also construct a confidence ball by using the posterior $\tilde{\pi}(\vartheta|X)$ 
	defined by \eqref{emp_norm_posterior} with the same resulting properties, 
	but with more involved mathematical derivations.
	Property (i) of Theorem \ref{th2} means that $\hat{r}^2$ is at most of the order of 
	the variance part of the oracle rate $r^2(\theta)$, so the size property 
	holds uniformly over $\theta\in\mathbb{R}^n$. 
	But this goes at the expense of the coverage, namely, the coverage property does not hold uniformly. 
	
	Indeed, according to Theorem \ref{th2}, 
	$\varrho\sigma^2 |I_*|\log(en/|I_*|) -M \sigma^2 \le \hat{r}^2 \le M'_0 \sigma^2 |I_o|\log(en/|I_o|)+M \sigma^2$ 
	with large probability. But this shell can be wide if  
	$\sigma^2 |I_*|\log(en/|I_*|)\ll \sigma^2 |I_o|\log(en/|I_o|)$. 
	If this happens (for deceptive $\theta$'s), then the coverage 
	property of the ball $B(\check{\theta}, M\hat{r})$ cannot be guaranteed because its radius 
	can be of a smaller order than the oracle rate $r^2(\theta)$. 
	This problem will not occur for those (non-deceptive) $\theta$'s
	for which the bias part of the rate $r^2(I_*,\theta)$ (see definition \eqref{oracle_I}) 
	is within a multiple of its variance part $\sigma^2 |I_*|\log(en/|I_*|)$. 
	Indeed, then $\sigma^2 |I_*|\log(en/|I_*|)$ must be at least some multiple of 
	$r^2(I_*,\theta)$ which is in turn bigger than the oracle rate $r^2(\theta)$ by the definition 
	\eqref{oracle} of the oracle. This means that $\sigma^2 |I_*|\log(en/|I_*|)$ is at least 
	of the oracle rate order, which, together with (ii) of Theorem \ref{th2}, imply 
	that $\hat{r}$ is also at least of the oracle rate order, resulting in a good coverage 
	of the confidence ball $B(\check{\theta}, M_2\hat{r}+M\sigma^2)$  for some $M_ 2$ 
	and sufficiently large $M$. This discussion motivates introducing the following condition. 
	\smallskip
	
	\noindent
	{\sc Condition EBR.}
	We say that $\theta\in\mathbb{R}^n$ satisfies the 
	\emph{excessive bias restriction} (EBR) condition with structural parameter $t\ge 0$ if  
	$\theta \in \Theta_{\rm eb}(t)$, where the corresponding set (called the \emph{EBR class}) is 
	\begin{align}
		\label{cond_ebr}
		\Theta_{\rm eb}(t)=\Theta_{\rm eb}(t,\tau_0)
		=\Big\{\theta\in\mathbb{R}^n : 
		\textstyle{\sum_{i\in I_*^c}} \theta_i^2 \le t \sigma^2\big(1+|I_*|\log(\tfrac{en}{|I_*|})\big)\Big\},
	\end{align}
	where the set $I_*=I_o^{\tau_0}$ is defined by \eqref{def_I_*}. 
	
	The condition EBR essentially requires that the bias part of the rate $r^2(I_*,\theta)$ 
	is dominated by a multiple of its variance part (additional  $\sigma^2$ is needed to handle 
	the case $I_*=\varnothing$).   
	This is obviously satisfied also for the rate $r^2(I',\theta)$ for all $I'\in\mathcal{I}_o$ such that 
	$I_* \subseteq I'$ (hence also for the oracle $I_o$ since $I_* \subseteq I_o$), 
	where the family $\mathcal{I}_o$ is defined by \eqref{family_I_o}.
	Besides, $\Theta_{\rm eb}(t_1) \subseteq \Theta_{\rm eb}(t_2)$ for $t_1 \le t_2$, 
	and, by the definition of $I_*$, $\mathbb{R}^n = \Theta_{\rm eb}(\tau_0n)$.

	


	Now we can use  the center $\hat{\theta}$ and the radius $\hat{r}$ in 
	constructing a confidence ball for $\theta$. 
	The following theorem, which is the main result in the paper, describes 
	the coverage and size properties of the confidence ball based on
	$\hat{\theta}$ and $\hat{r}$. 
	
	\begin{theorem}
		\label{th8}
		Let Conditions \eqref{cond_nonnormal} and \eqref{cond_technical} be fulfilled. 
		Then there exist constants $M_2, H_2, m_2>0$  such that for any $t,M\ge 0$, and with 
		$\hat{R}^2_M=\hat{R}^2_M(M_2)=(t+1)M_2\hat{r}^2+(t+2)M\sigma^2$,
		\begin{align*}
			&\sup_{\theta\in\Theta_{\rm eb}(t)} \mathrm{P}_{\theta}
			\big(\theta\notin B(\hat{\theta},\hat{R}_M)\big) 
			\le H_2 e^{-m_2M},\; 
			\sup_{\theta\in\mathbb{R}^n}  \mathrm{P}_{\theta}\big(\hat{r}^2\ge 
			M'_0\sigma^2|I_o | \log(\tfrac{en}{|I_o|})\!+\!(M\!+\!1)\sigma^2\big)
			\le H'_0  e^{-M},
		\end{align*}
		where $\Theta_{\rm eb}(t)$ is defined by \eqref{cond_ebr}, 
		the constants  $M'_0,H'_0$ are defined in Theorem \ref{th2}.
	\end{theorem}
	
	\begin{remark}
		\label{rem4}
		Let the quantity $b(\theta)$ (called \emph{excessive bias ratio}) 
		be defined by
		\begin{align}
			\label{def_b}
			b(\theta)=b(\theta,\tau_0)=\frac{ \sum_{i\in I_*^c} \theta_i^2}
			{\sigma^2+\sigma^2|I_*|\log(en/|I_*|)}
			=\frac{ \sum_{i=i_*+1}^{n} \theta_{[i]}^2}
			{\sigma^2+\sigma^2i_*\log(en/i_*)}
			=\frac{B(I_*,\theta)}{\sigma^2+V(I_*, \theta)}.
		\end{align}
		Note that, when proving Theorem \ref{th8}, we actually established 
		the following local assertions: 
		there exist constants $M_2, \alpha_1,m''_1,H_2, m_2>0$  such that for any $\theta\in\mathbb{R}^n$ 
		and any $\alpha, M\ge 0$
		\begin{align*}
			&\mathrm{P}_{\theta}
			\big(\theta\notin B(\hat{\theta},[(b(\theta)+1)M_2\hat{r}^2+(b(\theta)+2)M\sigma^2]^{1/2}\big) \\
			& \qquad \qquad \le 
			H_1\big(\tfrac{en}{ |I_o|}\big)^{-\alpha_1|I_o|} e^{-m_1M}
			+H'_1 \big(\tfrac{en}{ |I_*|}\big)^{-\alpha'_1|I_*|}e^{-m''_1M}
			\le H_2 e^{-m_2M},\\
			&\mathrm{P}_{\theta}\big(\hat{r}^2\ge 
			\sigma^2 (M'_0+\alpha) |I_o | \log(\tfrac{en}{|I_o|})+(M+1)\sigma^2\big)
			\le H'_0 (\tfrac{ne}{|I_o|})^{-\alpha|I_o|} e^{-M},
		\end{align*}
		where all the other constants ($H_1, m_1, H'_1, \alpha'_1, M'_0,H'_0$) are 
		defined in Theorems \ref{th1} and \ref{th2}.
		Notice that the above size relation holds uniformly in $\theta\in\mathbb{R}^n$.
		Although the coverage relation is also uniform in $\theta\in\mathbb{R}^n$,
		the main (unavoidable) problem is the dependence of the coverage relation on
		$b(\theta)$. That is why we introduced the EBR condition which essentially 
		provides control over the quantity $b(\theta)$.  
	\end{remark}
	
	\begin{remark}
		\label{rem5}
		The smaller constant $\tau_0$ (involved in the definition of the EBR condition) is, 
		the less restrictive the EBR condition is, the limiting case 
		$\tau_0\downarrow 0$ corresponds basically to no condition. 
		However, the main message here is that for any specific distribution of error vector $\xi$
		there is always some value of the constant $\tau_0$ in the EBR condition 
		(bounded away from zero, depending on how ``bad'' $\xi$ is).
		We treat a general situation and are not concerned with the most exact (smallest) 
		value for $\tau_0$, our bound for $\tau_0$ is in terms of $\beta,B$ and possibly too conservative 
		for each specific distribution of $\xi$. 
	\end{remark}
	
	The idea of the set $I_*=I_o^{\tau_0}$ introduced by \eqref{def_I_*} is that it contains $i_*=|I_*|$ most significant  
	coordinates of vector $\theta$ that are (essentially) not missed by the 
	procedure $\hat{I}$.
	The bias term of the rate $r^2(I_*,\theta)$ is the error that is made when 
	setting significant coordinates to zero (whereas they may not be zero).  
	Large ratio $b(\theta)$ defined by \eqref{def_b} means that this error 
	is relatively large as compared to the variance part of the rate $r^2(I_*,\theta)$. In a way,
	such $\theta$'s ``trick'' the procedure $\hat{\theta}$ and can therefore be regarded as deceptive.
	For each $\theta\in\mathbb{R}^n$, $b(\theta)$ measures the amount of deceptiveness of $\theta$: 
	the bigger $b(\theta)$, the more deceptive $\theta$.
	The EBR condition says that the deceptiveness has to be restricted:
	$\Theta_{\rm eb}(t) = \{\theta\in\mathbb{R}^n: b(\theta) \le t\}$.
	An explicit example of EBR parameters is 
	the set of \emph{self-similar} parameters introduced in 
	\cite{vanderPas&Szabo&vanderVaart:2017} which is in our terms 
	$\Theta_{\mathrm{ss}}(p,c,\tau_0)=\{\theta\in \ell_0[p]: |I_*(\theta)| \ge c p\}$
	for $p\in[n]$, $c\in(0,1]$. If $\theta \in \Theta_{\mathrm{ss}}(p,c,\tau_0)$,  then
	$p\le c^{-1}|I_*|$ and 
	$\sum_{i\in I_*^c} \theta_i^2 
	\le\sum_{|I_*|}^{c^{-1} |I_*|} \theta^2_{[i]} \le 
	(c^{-1}-1)\tau_0 \sigma^2 |I_*|\log(\tfrac{en}{|I_*|})$, where the second inequality follows
	by the oracle definition. Hence,  
	$\Theta_{\rm{ss}}(p,c,\tau_0)\subseteq \Theta_{\rm eb}((c^{-1}-1)\tau_0)$. 
	Notice that $\Theta_{\rm{ss}}(p,c,\tau)\subseteq \Theta_{\rm eb}((c^{-1}-1)\tau)$ 
	for any $\tau>0$.
	
	In particular, for $\theta \in \Theta_{\mathrm{ss}}(p,1,\tau_0)=\{\theta\in\ell_0[p]: 
	\, |I_*(\theta)|=p\}$,  
	the insignificant coordinates  $I_*^c$ of such $\theta$'s are the true zeros
	and the significant coordinates $I_*$ are sufficiently distinct from zero.
	Then the set $I_*(\theta)$ coincides with the support $\text{supp}(\theta)\triangleq \{i\in[n]: \theta_i \not =0\}$, 
	i.e., $I_*(\theta)=\text{supp}(\theta)$, so that $B(I_*,\theta)=0$, implying 
	$\Theta_{\mathrm{ss}}(p,1,\tau_0) \subseteq \Theta_{\rm eb}(0)$. 
	This class consists of the ``nicest'' (least deceptive) parameters satisfying 
	the EBR condition with zero deceptiveness $t=0$.
	The uncertainty quantification result is the strongest for this class because 
	the inflating factor is the smallest as $t=0$.
	More about the EBR condition is in Section \ref{subsec_ebr}. 

	\subsection{Confidence ball of $n^{1/4}$-radius without EBR}
	By analyzing the previous results, we see that the resulting radius $\hat{r}$ of our constructed 
	confidence ball is of the oracle rate only under the EBR condition. In general, 
	$\hat{r}^2$ underestimates the oracle rate $r^2(\theta)$. The difference is the bias 
	term which may in general  be of a bigger order than the variance part, leading to a bad coverage. 
	Suppose we want to construct a confidence ball of a full coverage uniformly 
	over the whole space $\mathbb{R}^n$.
	Recall however that, in view of the above mentioned negative results of \cite{Li:1989},
	\cite{Cai&Low:2004}, \cite{Baraud:2004} and \cite{Nickl&vandeGeer:2013}, no data dependent 
	ball can provide full coverage and adaptive size simultaneously. Insisting on the full coverage,  
	one can at best adapt to sparsity levels only in the range $s\ge C\sqrt{n}$ 
	(i.e., actually for non-sparse parameters) and the term of order $\sigma^2\sqrt{n}$ should 
	be in the radius. 
	Let us give a heuristics behind this.
	An idea is to mimic the quantity $\|\theta-\hat{\theta}\|^2$ by $\hat{R}^2=\|X-\hat{\theta}\|^2$.
	Clearly, there is a lot of bias in $\hat{R}^2$, the biggest part of which is due to the term 
	$\sigma^2\|\xi\|^2$ contained in $\hat{R}$. To de-bias for that part, we need to subtract its expectation 
	$\sigma^2\mathrm{E}\|\xi\|^2 = n \sigma^2$, 
	where we assumed $\mathrm{Var}(\xi_i)=1$ in the model \eqref{model} for simplicity. 
	However, even de-biased quantity $\hat{R}^2$ can only be controlled up to a margin 
	of the order $\sigma^2 \sqrt{n}$. That is why a term of the order $\sigma n^{1/4}$ is necessary in the radius of the confidence ball to provide coverage uniformly over the whole space $\mathbb{R}^n$.
	
	
	To handle some technical issues 
	in this case, we impose the following additional condition.\smallskip\\
	{\sc Condition (A3).} Besides $X$ given by \eqref{model}, we also observe 
	$X'\in\mathbb{R}^n$ independent of $X$, where $X'=\theta+ \sigma \xi'$, 
	the random vector $\xi'$ satisfies the following relations:
	$\mathrm{E} \xi'_i =0$, $\mathrm{Var}(\xi'_i)\le C_\xi$, $ i\in[n]$; 
	\begin{align}
		\mathrm{P}\big(|\langle v, \xi' \rangle|\ge \sqrt{M}\big)\le\psi_1(M)\;\;
		\forall\, v\in\mathbb{R}^n: \,  
		\|v\| =1; \;
		\label{cond_A3}
		\tag{A3}
		\mathrm{P} \big(\big|\|\xi'\|^2- \mathrm{E} \|\xi'\|^2\big| \ge M \sqrt{n}\big) \le \psi_2(M).
	\end{align}
	Here $\psi_1(M), \psi_2(M)$ are some positive monotonically decreasing functions such that 
	$\psi_1(M)\downarrow 0$ and $\psi_2(M)\downarrow 0$ as $M\uparrow \infty$. \smallskip\\
	Condition (A4) is satisfied for independent normals 
	$\xi_i \overset{\rm ind}{\sim} \mathrm{N}(0,1)$ even if we do not have the sample $X'$ at our 
	disposal. Indeed, in this case we can ``duplicate'' the observations by randomization
	at the cost of doubling the variance in the following manner: 
	create samples $X' = X+\sigma Z$ and $X''=X-\sigma Z$, for a
	$Z=(Z_1,\ldots,Z_n)$ (independent of $X$) such that $Z_i\overset{\rm ind}{\sim} \mathrm{N}(0,1)$. 
	Relations \eqref{cond_A3} are then fulfilled with exponential  functions
	$\psi_l(M)=Ce^{-c M}$, $l=1,2$, for some $C,c>0$. 
	If the sub-gaussianity condition \eqref{subgaussian} is fulfilled for $\xi'$ 
	(which is the same as Condition (A1) in case of independent $\xi'_i$'s), then 
	$\psi_1(M) = e^{-\rho M}$.
	By Chebyshev's inequality, we see that the second relation in \eqref{cond_A3} 
	is fulfilled with function $\psi_2(M) = cM^{-2}$ for any zero mean independent 
	$\xi'_i$'s with 
	$\mathrm{E} \xi_i'^4\le C$. 
	
	Coming back to the problem of constructing a confidence ball of full coverage 
	uniformly over $\mathbb{R}^n$, let $\hat{\theta}$ and $\hat{I}$ be defined as before 
	and based on the sample $X$. We propose to mimic $\|\theta-\hat{\theta}\|^2$ by the
	de-biased quantity $\|X'-\hat{\theta}\|^2-n\sigma^2$ 
	plus additional $\sigma^2 \sqrt{n}$-order term to control its oscillations, 
	leading us to the following data dependent radius
	\begin{align}
		\label{radius2}
		\tilde{R}^2_M=\big(\|X'-\hat{\theta}\|^2-n\sigma^2 
		+2\sigma^2G_M\sqrt{n}
		\big)_+,
		\quad \text{where} \quad G_M=\sqrt{M(M+M_1)},
	\end{align}
	$x_+ =x\vee 0$ and the constant $M_1$ is from Theorem \ref{th1}.
	The next theorem establishes the coverage and size properties of the confidence ball 
	$B(\hat{\theta},\tilde{R}_M)$. 
	\begin{theorem}
		\label{th4}
		Let Conditions \eqref{cond_nonnormal}, \eqref{cond_technical} and \eqref{cond_A3} be fulfilled 
		and $\tilde{R}^2_M$ be defined by \eqref{radius2}. Then for any $M\ge 0$
		\begin{align*}
			\sup_{\theta\in\mathbb{R}^n} \mathrm{P}_{\theta}
			\big(\theta\notin B(\hat{\theta},\tilde{R}_M)\big) 
			&\le \psi_1(M/4)+\psi_2(M)+H_1e^{-m_1M},\\
			\sup_{\theta\in\mathbb{R}^n}  \mathrm{P}_{\theta}\big(\tilde{R}_M^2\ge g_M(\theta,n)\big)
			&\le\psi_1(M/4)+\psi_2(M)+2H_1e^{-m_1M},
		\end{align*}
		$g_M(\theta,n)=M_1r^2(\theta)+M\sigma^2+4\sigma^2 G_M \sqrt{n}$ 
		and the constants $H_1, m_1, M_1$ are defined in Theorem \ref{th1}.
	\end{theorem}
	By taking large enough $M$ we can ensure the coverage and size relations uniformly 
	over the entire space $\mathbb{R}^n$. However, notice the price for this overall uniformity: the radius 
	of the constructed confidence ball is essentially of the order $\sigma n^{1/4}+ r(\theta)$. So, it is always 
	of the order at least  $\sigma n^{1/4}$ even for very sparse parameters $\theta$, 
	and it is of the oracle rate order only for non-sparse parameters, when $r(\theta) \ge C\sigma n^{1/4}$.
	This is a fundamental problem for uncertainty quantification, which typically occurs when the parameter $\theta$ 
	has many smallish coordinates $\theta_i$, say, with $\theta_i^2$  just under the noise level $\sigma^2$. 
	Clearly, in this case no method can reliably assign those coordinates to the significant set. 
	As demonstrated in \cite{Li:1989}, \cite{Cai&Low:2004}, \cite{Baraud:2004} and \cite{Nickl&vandeGeer:2013}, the
	above mentioned price in the form of a big radius is absolutely unavoidable (even in the case of just two sparsity 
	classes as is shown in  \cite{Nickl&vandeGeer:2013}), as soon as we require uniform coverage.

	\subsection{Implications: the minimax results over sparsity classes}
	\label{implication}
	In this subsection we elucidate the potential strength of the local approach.
	In particular, we demonstrate how the global adaptive 
	minimax results over certain scales can be derived from the local results. 
	Note that the oracle rate $r(\theta)$ is a local quantity in that it quantifies the level of accuracy
	of inference about specific $\theta$ and originally it is not linked to any particular scale of  
	classes. However, it is always possible to relate the oracle rate to various scales. 
	Precisely, if we want to establish global adaptive 
	minimax results over certain scale, say, $\{\Theta_\beta, \, \beta\in\mathcal{B}\}$, 
	with corresponding minimax rates $\{r(\Theta_\beta),\beta\in\mathcal{B}\}$ 
	(the minimax rate over $\Theta_\beta$ is $r^2(\Theta_\beta) \triangleq \inf_{\hat{\theta}} 
	\sup_{\theta \in \Theta_\beta}\mathrm{E}_{\theta} \|\hat{\theta}-\theta\|^2$, 
	where the infimum is taken over all estimators), the only thing we need to show is
	\[
	\sup_{\theta\in\Theta_\beta}r^2(\theta)\le 
	c r^2(\Theta_\beta),\quad \text{for all} \;\; \beta\in\mathcal{B}.
	\]
	If the above property holds, we say the oracle rate $r(\theta)$ \emph{covers} the scale 
	$\{\Theta_\beta, \, \beta\in\mathcal{B}\}$. In this case, the local results 
	on the estimation, the posterior contraction and the size relation of 
	the confidence ball will immediately imply the corresponding global adaptive minimax 
	results over the covered scale, (actually, simultaneously for all scales that are covered 
	by the oracle rate $r(\theta)$). 
	As to the coverage property, according to Theorem \ref{th8}, it holds uniformly 
	only over the EBR class $\Theta_{\rm eb}(t)$, whichever scale we consider.
	Thus, specializing the coverage property to a particular scale boils down to  
	intersecting this scale with the EBR class $\Theta_{\rm eb}(t)$ in the 
	coverage property.
	
	Next, we consider two sparsity scales $\{\Theta_\beta, \beta\in\mathcal{B}\}$ 
	for which the adaptive minimax results 
	(on the estimation problem, the contraction rate of the empirical Bayes posterior, and the size property
	of the confidence ball $B(\hat{\theta},(M_2\hat{r}^2+M)^{1/2})$)
	follow from our local results Theorems \ref{th1} and \ref{th8}.  
	The results for other (covered) scales can also be readily derived.

	\paragraph{Nearly black vectors.} For $p_n\in[n]$ such that   
	$ p_n=o(n)$ as $n\to \infty$ 
	(we use the usual $o$, $O$ notation to describe the asymptotic behavior of 
	certain quantities as $n\to \infty$), introduce the sparsity class
	$\ell_0[p_n]=\{\theta\in\mathbb{R}^n: s(\theta)=|I^*(\theta)|\le p_n\}$, where 
	by $I^*(\theta)$ and $s(\theta)$ we denote the active index set and 
	the sparsity of $\theta\in\mathbb{R}^n$:
	\begin{align}
		\label{def_s(theta)}
		I^*(\theta)=\{i\in[n]: \theta_i \not =0\}, \qquad  s(\theta)=|I^*(\theta)|.
	\end{align}

	The minimax estimation rate over the class  of nearly black vectors $\ell_0[p_n]$ with the sparsity parameter 
	$p_n$  is known to be $r^2(\ell_0[p_n])=O\big(\sigma^2p_n\log(\frac{n}{p_n})\big)$ 
	as $n\to\infty$; see \cite{Donoho&Johnstone&Hoch&Stern:1992}. 
	By the definition  \eqref{oracle} of the oracle rate $r^2(\theta)$,
	we have that $r^2(\theta) \le r^2(I^*(\theta),\theta)$. Then 
	we obtain trivially that  
	\[
	\sup_{\theta\in \ell_0[p_n]} r^2(\theta) 
	\le \sup_{\theta\in \ell_0[p_n]} r^2(I^*(\theta),\theta)
	\le \sigma^2 p_n \log\big(\tfrac{en}{p_n}\big)=O\big(r^2(\ell_0[p_n])\big).
	\]
	The last relation, Theorems  \ref{th1} and \ref{th8}  immediately  imply
	the adaptive minimax results for the scale $\ell_0[p_n]$. 
	We summarize these results  in the following corollary.
	\begin{corollary}
		\label{theorem_black_vectors}
		Under the conditions of Theorem \ref{th8}, we have for any $M\ge 0$
		\begin{align*}
			&\sup_{\theta\in\ell_0[p_n] }\mathrm{E}_{\theta} 
			\hat{\pi}\big( \|\vartheta-\theta\|^2\ge M_0 \sigma^2 p_n\log(\tfrac{en}{p_n}) +M\sigma^2 | X\big)  \le 
			H_0 e^{-m_0M}, \\
			&\sup_{\theta\in \ell_0[p_n]}\mathrm{P}_{\theta} 
			\big(\|\hat{\theta}-\theta\|^2 \ge M_1 \sigma^2 p_n\log(\tfrac{en}{p_n})+M\sigma^2\big)
			\le H_1 e^{-m_1 M},\\
			&\sup_{\theta\in \ell_0[p_n]}
			\mathrm{P}_{\theta}\big(\hat{r}^2\ge M'_0\sigma^2p_n \log(\tfrac{en}{p_n}) +(M+1)\sigma^2\big)\le
			H'_0 e^{-M}.
		\end{align*}
	\end{corollary}
	
	
	The next assertion describes some ``over-dimensionality''  
	(or ``undersmoothing'') control of the empirical Bayes posterior 
	$\hat{\pi}(I|X)$  from the $\mathrm{P}_{\theta}$-perspective. 
	\begin{theorem}
		\label{cor1_dimension}
		Let $s(\theta)$ be defined by \eqref{def_s(theta)}.
		Under the conditions of Theorem \ref{th2}, there exist $M_4, m_4>0$ such 
		that  for any $M>M_4$ and $\theta\in\mathbb{R}^n$   
		\[
		\mathrm{E}_{\theta}\hat{\pi}(I: |I|>M s(\theta)|X)
		\le
		C_0\exp\big\{-m_4s(\theta)\big[(M-M_4) \log(\tfrac{en}{s(\theta)}) -M\log M\big]\big\}.
		\]
		In particular, there exist constants $M'_4, m'_4>0$ such that
		\[
		\mathrm{E}_{\theta}\hat{\pi}(I: |I|>M'_4 s(\theta)|X)
		\le
		C_0\exp\big\{-m'_4s(\theta) \log(\tfrac{en}{s(\theta)})\big\}.
		\]
	\end{theorem}
	The above assertion is a local type result, but can readily be specialized to the sparsity 
	class $\theta\in\ell_0[s_n]$ in the minimax sense.  
	If $s(\theta) \ge 1$, the probability bound goes to $0$ as $n\to \infty$.

	\paragraph{Weak $\ell_s$-balls.}
	For $s\in(0,2)$, the \emph{weak $\ell_s$-ball} with the sparsity parameter $p_n$ is defined by
	\[
	m_s[p_n]=\big\{\theta\in\mathbb{R}^n: 
	\theta^2_{[i]}\le(p_n/n)^2(n/i)^{2/s}, \, i\in\mathbb{N}_n\big\}, \quad 
	p_n=o(\sigma n) \;\; \text{as} \;\; n\to \infty,
	\] 
	where $\theta_{[1]}^2\ge\ldots\ge \theta_{[n]}^2$ are the ordered 
	$\theta_1^2,\ldots, \theta_n^2$. This scale can be thought of as Sobolev hyper-rectangle
	for ordered (with unknown locations) coordinates:
	$m_s[p_n]=\mathcal{H}(\beta,\delta_n)=
	\{\theta\in\mathbb{R}^n: |\theta_{[i]}| \le \delta_n i^{-\beta}\}$, with 
	$\delta_n =n^{1/s} \tfrac{p_n}{n}$ and $\beta=1/s > 1/2$.

	Denote $j=O_\theta(i)$ if $\theta_i^2=\theta_{[j]}^2$, with 
	the convention that in the case $\theta_{i_1}^2=\ldots =\theta_{i_k}^2$ for $i_1 < \ldots < i_k$ 
	we let $O_\theta(i_{l+1})=O_\theta(i_l)+1$, $l=1,\ldots, k-1$. 
	The minimax estimation rate  over this class is 
	$r^2(m_s[p_n])=n(\tfrac{p_n}{n})^s[\sigma^2\log(\tfrac{n\sigma}{p_n})]^{1-s/2}$ 
	when $n^{2/s}(\tfrac{p_n}{n})^2 \ge \sigma^2 \log n$, and $r^2(m_s[p_n])
	=n^{2/s}(\tfrac{p_n}{n})^2 +\sigma^2$ when $n^{2/s}(\tfrac{p_n}{n})^2<\sigma^2 \log n$,
	as $n\to \infty$; see \cite{Donoho&Johnstone:1994b} and  \cite{Birge&Massart:2001}. 
	Then take $I^*(\theta)=\{i \in \mathbb{N}_n: O_\theta(i) \le p^*_n\}$,
	with $p_n^*=en (\frac{p_n}{n\sigma})^s[\log(\frac{n\sigma}{p_n})]^{-s/2}$ 
	in the case $n^{2/s}(\tfrac{p_n}{n})^2\ge\sigma^2 \log n$, to derive  
	\begin{align}
		\label{weak_balls_minimax_risk}
		&\sup_{\theta\in m_s[p_n]}  r^2(\theta)\le \sup_{\theta\in m_s[p_n]} r^2(I^*(\theta),\theta)
		\le \sigma^2 p_n^*\log(\tfrac{en}{p^*_n})+n^{2/s} (\tfrac{p_n}{n})^2\sum_{i>p_n^*} i^{-2/s} 
		\notag\\
		&\le K_1\sigma^2p_n^*\log(\tfrac{n\sigma}{p_n})+K_2 n^{2/s} (\tfrac{p_n}{n})^2 (p^*_n)^{1-2/s} 
		\le Kn (\tfrac{p_n}{n})^s\big[\sigma^2 \log (\tfrac{n\sigma}{p_n})\big]^{1-s/2}
		=O\big(r^2(m_s[p_n])\big),  
	\end{align}
	for some $K=K(s)$.
	The case $n^{2/s}(\tfrac{p_n}{n})^2<\sigma^2 \log n$ is treated similarly by taking
	$p_n^*=0$. 
	
	Theorems  \ref{th1} and \ref{th8}  
	imply the minimax adaptive results for the scale $m_s[p_n]$.
	\begin{corollary}
		\label{theorem_weak_balls}
		Under the conditions of Theorem \ref{th8}, we have for any $M \ge 0$
		\begin{align*}
			&\sup_{\theta\in m_s[p_n]}\mathrm{E}_{\theta} 
			\hat{\pi}\big( \|\vartheta-\theta\|^2\ge M_0 K r^2(m_s[p_n])+M\sigma^2|X\big)
			\le H_0e^{-m_0M},\\
			&\sup_{\theta\in m_s[p_n]}\mathrm{P}_{\theta}\| 
			\big(\hat{\theta}-\theta\|^2
			\ge M_1K r^2(m_s[p_n])+M\sigma^2\big) \le H_1e^{-m_1M}, \\
			&\sup_{\theta\in m_s[p_n]}
			\mathrm{P}_{\theta} \big(\hat{r}^2\ge M'_0 K r^2(m_s[p_n])+(M+1)\sigma^2
			\big)\le H'_0 e^{-M}.
		\end{align*}
	\end{corollary}
	
	The following claim concerns the ``over-dimensionality'' control for the  class $m_s[p_n]$.
	\begin{theorem}
		\label{cor2_dimension}
		Let the conditions of Theorem \ref{th2} be fulfilled, 
		$p_n^*=en(\tfrac{p_n}{n\sigma})^s \big[\log(\tfrac{n\sigma}{p_n})\big]^{-s/2}$ 
		and $n^{2/s}(\tfrac{p_n}{n})^2 \ge \sigma^2 \log n$.
		Then there exist constants $M_5,m_5>0$ such that for any $M>M_5$ 
		there exists $n_0=n_0(M,s)$ such that, for all $n \ge n_0$,  
		\[
		\sup_{\theta\in m_s[p_n]} \mathrm{E}_{\theta}\hat{\pi}(I: |I|>Mp_n^*|X)
		\le C_0 \exp\big\{-m_5(M-M_5)p_n^*\log (\tfrac{n\sigma}{p_n})\big\}.
		\]
	\end{theorem}
	Notice that the exponential upper bound from the last relation converges to zero as $n\to \infty$
	because $p_n^*\log (\tfrac{n\sigma}{p_n}) \ge e (\sigma^2\log n)^{s/2} (\log (\tfrac{n\sigma}{p_n})^{1-s/2}$.
	
	\begin{remark}
		The same minimax results hold over the so called \emph{strong $\ell_s$-ball}
		$\ell_s[p_n]=\{\theta\in\mathbb{R}^n:\, \tfrac{1}{n} \sum_{i=1}^n |\theta_i|^s \le (\tfrac{p_n}{n})^s \}$, 
		$s \in (0,2)$, since $\ell_s[p_n] \subseteq  m_s[p_n] \subseteq \ell_{s'}[p_n]$ for any $s'>s$. 
	\end{remark}
	
	\section{Concluding remarks and EBR} 
	\label{sec_discussion_and_EBR}
	
	\subsection{Concluding remarks}
	\label{subsec_duscussion}
	\paragraph{Improving constants.} 
	Since our approach applies to a very general situation, many constants involved in 
	the conditions and proofs may be rather conservative. Indeed, we do not specify 
	any distribution of $\xi$ and even dot not assume independence of its coordinates. 
	For the problem to be at all solvable, the vector $\xi$ has to have some minimal structure 
	which is in our case provided by Condition \eqref{cond_nonnormal}. 
	The constants $\beta,B$ reflect in a generic way 
	how bad (or how good) the vector $\xi$ is,
	implying that almost all the constants in the proofs and conditions depend on $\beta,B$. 
	Clearly, if a distribution of $\xi$ is specified, many bounds can made more precise 
	and many constants can be improved, including the constants $\bar{\varkappa}$ and 
	$\tau_0$ from Conditions \eqref{cond_technical} and (A3), 
	see Remark \ref{rem5} for more on constant $\tau_0$.
	Besides, some constants can be improved by using more precise inequalities 
	at some steps of the proof. But this would make the presentation significantly 
	lengthier without adding anything new conceptually.


	For example, in case $\xi_i \overset{\rm ind}{\sim} \mathrm{N}(0,1)$, we can sharpen 
	up many constants in the proofs and conditions. In the proof of  Lemma \ref{lemma1}, 
	we can compute exactly the right hand side of \eqref{relation_P2} by using the 
	elementary identity: for $Y\sim \mathrm{N}(\mu_y,\sigma_y^2)$,
	\begin{align}
		\label{elem_inequality_1}
		\mathrm{E}\exp\big\{\tfrac{aY^2}{2}\big\}=
		\exp\big\{\tfrac{a\mu_y^2}{2(1-a\sigma_y^2)}-\tfrac{1}{2}\log(1-a\sigma_y^2)\big\},
		\quad \text{for any} \;\; a<\sigma_y^{-2}.
	\end{align}
	By some tedious but straightforward calculations, we obtain the claim of 
	Lemma~\ref{lemma1} for any $h \in [0,1)$ with the constants
	$A_h=\tfrac{h}{2(1+h)}$, $B_h=\tfrac{h}{2(1-h)}$, $C_h=\tfrac{h}{2}$ and 
	$D_h=\tfrac{h}{2}+\tfrac{1}{2}\log (1-h)$. If $I\backslash I_0 = \varnothing$, the bound 
	holds also for $h=1$ with $A_1=\frac{1}{4}$, $B_1=0$, $C_1= D_1=\frac{1}{2}$.
	Next, since Lemma \ref{lemma1} now holds for any $h \in [0,1)$, we can try to optimize 
	the choice of $h$ in Lemma \ref{lemma2}. 
	We can also relax the requirement $c_1 >2$ to $c_1>1$ in Lemma  \ref{lemma2}, leading 
	to the bound for $\varkappa \ge \bar{\varkappa}= 2.04$.

	The constants in the proof of Theorem \ref{th2} can also be improved 
	in the normal case
	and we can use the bound 
	$\mathrm{E}\big(\sum_{i\in I} \xi_i^2 \big)^2= |I|^2+2|I| \le 3|I|^2$
	instead of \eqref{moment_bound} in the proof of Theorem \ref{th1}.

	\paragraph{Product prior.}
	\label{subsec_product_prior}
	If, instead of the prior $\pi$, we take a prior $\bar\pi=\bar\pi _{K,\varkappa}= \sum_{I \in\mathcal{I}} \lambda_I \pi_I$ 
	with $\tau^2_i(I) =K\sigma^2\mathbbm{1}\{i\in I\}$ for any fixed $K>0$ (we can even allow 
	$K=K_n \to\infty$, but $K_n=O(n)$, as $n\to \infty$) in \eqref{gen_norm_prior} and 
	$\lambda_I= c_{\varkappa,n} \exp\{-\varkappa |I| \log n\}$ (with $\varkappa>\varkappa_0$ for 
	some $ \varkappa_0>0$) in \eqref{prior_lambda}, then all the results will hold with 
	$\log n$ instead of $\log(\tfrac{en}{|I|})$ in the oracle rate \eqref{oracle}. 
	This case was studied in the first version of the arXiv-preprint of 
	this paper. Thus, the results for the prior $\bar\pi$ are weaker than the results 
	obtained in this paper. For example, the minimax rates for the sparsity classes 
	(Corollaries \ref{theorem_black_vectors}, \ref{theorem_weak_balls}) 
	follow from these weaker results only if the sparsity parameter 
	$p_n = O(n^\gamma)$ for $\gamma\in[0,1)$ as $n\to \infty$, 
	otherwise we obtain only the \emph{near-minimax} rates, with the factor $\log n$ 
	instead of $\log(\tfrac{n}{p_n})$.
	
	However, there is an advantageous feature of the prior $\bar{\pi}$ as compared with $\pi$. 
	Namely, it is of the product structure: for $\lambda_I =c_\lambda \prod_{i\in I} \lambda_i$ 
	with $c_\lambda= \prod_{i=1}^n (1+\lambda_i)^{-1}$, we compute
	$\bar{\pi}=\sum_{I \in \mathcal{I}} \lambda_I \pi_I = 
	\bigotimes_{i=1}^n \big[\omega_i \mathrm{N}(\mu_{1,i},K\sigma^2)+ 
	(1-\omega_i)\delta_0\big]$,
	$\omega_i =\frac{\lambda_i}{1+\lambda_i}$
	($\omega_i= \lambda(i\in I)$   
	is the prior probability that the random set $I$ contains $i$).
	This leads to the product structure of  the empirical Bayes posterior,  so that
	the computation of the corresponding empirical Bayes estimator can easily be done 
	in the coordinatewise fashion. Indeed, in our case $\lambda_i=\lambda=n^{-\varkappa}$  
	and some computations give the following  empirical Bayes posterior
	\[
	\bar{\pi}(\vartheta|X) =\bigotimes_{i=1}^n \big[p_i\mathrm{N}\big(X_i, \tfrac{K \sigma^2}{K+1}\big) 
	+(1-p_i) \delta_0\big], \qquad p_i= 
	1/\big[1+h\exp\{-\tfrac{X_i^2}{2\sigma^2}\}\big],
	\]
	where $p_i= \bar{\pi}(\theta_i \not = 0|X)$ and $h=h_{\varkappa,K} =\tfrac{\sqrt{K+1}}{\lambda} = 
	n^\varkappa(K+1)^{1/2}$.
	The mean  with respect to $\bar{\pi}(\vartheta|X)$ is readily obtained:
	$\bar{\theta}=\mathrm{E}_{\bar{\pi}}(\vartheta|X) =(p_iX_i, \, i\in[n])$,  
	a shrinkage estimator with easily computable shrinkage factors $p_i$.
	Coordinatewise empirical Bayes medians can also be easily computed.

	\paragraph{Cardinality dependent prior $\lambda$.}
	\label{subsec_C_vdV}
	Notice that the prior $\lambda=(\lambda_I,\, I \in \mathcal{I})$ defined by \eqref{prior_lambda}
	depends  on the set $I \in\mathcal{I}$ only via its cardinality $|I|$, i.e.,
	$\lambda_I= g(|I|)$ for some nonnegative function $g(k)$, $k=0,1,\ldots, n$. 
	It is easy to see that in this case $\pi_n(k)=g(k)\binom{n}{k}$, $k=0,1,\ldots, n$, 
	determines the prior on the cardinality of $I$.
	Hence, the prior $\lambda_I$ can always be modeled in two steps: first draw 
	the random cardinality 
	$K$ according to the prior $\pi_n(k)$, and then given $K=k$, draw a random set $\mathcal{I}$ 
	uniformly from the family of all subsets of $\mathcal{I}$ of cardinality $k$.
	Such priors $\lambda$ are used in \cite{Castillo&vanderVaart:2012},  
	where the cardinality prior $\pi_n(k)$ can be taken to be a so called  ``complexity prior'' 
	$\pi_n(k)=\exp\{-ak \log(bn/k)\} $ for some $a,b>0$. Since 
	$e^{k \log(n/k)} \le \binom{n}{k} \le e^{k\log(ne/k)}$, the resulting prior mass
	$\lambda_I$ on $I$ is bounded below and above by expressions of the type 
	$\exp\{-a_1|I| \log(b_1n/|I|)\}$, resembling the prior \eqref{prior_lambda}.
	The condition on the complexity prior from \cite{Castillo&vanderVaart:2012}
	essentially corresponds to our condition $\varkappa>\bar{\varkappa}$ 
	for some $\bar{\kappa}>0$ (Condition \eqref{cond_technical}).

	\paragraph{Computing the estimators.}
	\label{subsec_comp_est}
	Note that the estimator \eqref{estimator} is a shrinkage estimator, and 
	the estimator \eqref{emp_emp_posterior} is a hard thresholding procedure. 
	Indeed, the estimator \eqref{estimator} is
	$\tilde{\theta}_i=p_i X_i$ where $p_i =\sum_{I:\, i \in I} \tilde{\pi}(I|X)$,
	and the estimator \eqref{emp_emp_posterior} is
	$\check{\theta}_i=X_i 1\big\{|X_i| \ge \check{t}\big\}$, where
	$\check{t}=|X_{[\check{k}]}|$, $|X_{[1]}|\ge \ldots \ge |X_{[n]}|$,
	and $\check{k}$ is the minimizer of 
	$\sum_{i=k+1}^n X^2_{[i]} + (2\varkappa+1)\sigma^2 k \log(en/k)$. 
	
	The thresholding procedure is easy to implement, whereas the values  $p_i$ in 
	the shrinkage procedure are more difficult to compute. It is demonstrated 
	in \cite{Castillo&vanderVaart:2012} how one can use the partial product structure 
	(in the model and in $\pi_I$, but not in $\lambda_I$) 
	to facilitate the computation of $p_i$'s. Other estimators can be considered, 
	for example, the coordinatewise median with respect to $\tilde{\pi}$, 
	which is going to be something in between shrinkage and thresholding. 
	
	\paragraph{Condition \eqref{cond_nonnormal}.}
	\label{subsec_condition_A1}
	First we mention that all the results still hold, if, instead of  
	Condition \eqref{cond_nonnormal}, we assume  the weaker condition:
	$\mathrm{E} \exp\{\beta\sum_{i\in I} \xi_i^2\} \le C_\beta e^{B |I| \log (en/|I|)}$ for all 
	$I\in  \mathcal{I}$ and some $\beta\in (0,1]$, $B,C_\beta>0$. However, 
	we leave Condition \eqref{cond_nonnormal} in its present form
	to provide a cleaner mathematical exposition. 
	
	It is interesting to relate Condition \eqref{cond_nonnormal} to the so called 
	\emph{sub-gaussianity} condition on the error vector $\xi=(\xi_i, i \in[n])$. 
	The  random vector $\xi$ is called 
	\emph{sub-gaussian} with parameter $\rho >0$ if 
	\begin{align}
		\label{subgaussian}
		\mathrm{P}( |\langle v, \xi \rangle|> t) \le e^{-\rho t^2}\quad 
		\text{for all}\;\; t\ge 0\;\; \text{and all}\;\;  v\in\mathbb{R}^n \;\;  \text{such that} \;\;  \|v\| =1.
	\end{align}
	The sub-gaussianity condition \eqref{subgaussian} and Condition \eqref{cond_nonnormal} 
	are close, but in general incomparable. For example, let $\xi_i = \xi_0$, $i\in[n]$, 
	for some bounded random variable $\xi_0$ (say, uniform on $[-1,1]$), then Condition (A1) trivially holds 
	whereas the sub-gaussianity condition is not fulfilled.
	If the $\xi_i$'s are independent, then the sub-gaussianity 
	condition is equivalent to Condition \eqref{cond_nonnormal}. 
	
	

	In a way, Condition \eqref{cond_nonnormal} prevents too much dependence, but 
	it still allows some interesting cases of dependent $\xi_i$'s.  
	Suppose that the $\xi_i$'s follow an autoregressive model AR(1) with normal white noise: 
	\[
	\xi_k=\alpha \xi_{k-1} +\epsilon_k, \quad
	\epsilon_k\overset{\rm ind}{\sim} \mathrm{N}(0,1), \quad k\in[n]; 
	\quad \xi_0=0,\quad |\alpha|<1.
	\] 
	Let us show that Condition \eqref{cond_nonnormal} 
	is fulfilled for the vector $\xi=(\xi_i, i \in[n])$.
	We have that for any $k>l$, $\xi_k= \alpha^{k-l} \xi_{l}+\alpha^{k-l-1}
	\epsilon_{l+1} +\ldots +\epsilon_{k} = \alpha^{k-l} \xi_l + Z_{k-l}$,
	where $Z_{k'} \sim \mathrm{N}(0, \sigma^2_{k'})$ with 
	$ \sigma^2_{k'} = 1+\alpha^2 +\ldots +  \alpha^{2(k'-1)}\le \tfrac{1}{1-\alpha^2} 
	\triangleq \sigma_0^2$. Clearly, for any $I\in \mathcal{I}$, there are 
	$1\le k_1<k_2 <\ldots <k_{|I|}\le n$ such that 
	$\sum_{i\in I} \xi^2_i = \sum_{i=1}^{|I|} \xi^2_{k_i} $. 
	Denote $\mathcal{F}_m= \sigma(\xi_{k_i}, 1\le i \le m)$, $m\in[|I|]$, the 
	$\sigma$-algebra generated by $\{\xi_{k_i}, 1\le i \le m\}$.
	Choose $\beta$ and $\alpha$ in such a way that 
	$0<\tfrac{2\beta \alpha^2}{1-4\beta \sigma^2_0}\le \beta$.
	By using the elementary identity \eqref{elem_inequality_1}, we first evaluate
	the conditional expectation 
	\begin{align*}
		&\mathrm{E}  \big(e^{\beta(\xi^2_{k_{m-1}}+\xi^2_{k_m})} | \mathcal{F}_{m-1}\big)
		\le
		e^{\beta\xi^2_{k_{m-1}}}\mathrm{E}\big( e^{2\beta \xi^2_{k_m}}| \mathcal{F}_{m-1}\big) \\
		&=\exp\big\{\big(\beta+\tfrac{2\beta \alpha^{2(k_m - k_{m-1})}}{1-4\beta \sigma^2_{k_m,k_{m-1}}}
		\big)\xi^2_{k_{m-1}}-\tfrac{1}{2}\log(1-4\beta \sigma^2_{k_m,k_{m-1}}) \big\}
		\le (1-4\beta \sigma_0^2)^{-1/2} e^{2\beta\xi^2_{k_{m-1}}}.
	\end{align*}
	Iterating the above conditional expectation argument, we establish  Condition \eqref{cond_nonnormal}:
	\begin{align*}
		\mathrm{E}\exp\Big\{\beta\sum_{i\in I} \xi_i^2\Big\}&=
		\mathrm{E} \mathrm{E}\Big[ \exp\Big\{\beta\sum_{i\in I} \xi_i^2\Big\}\Big| \mathcal{F}_{|I|-1}\Big]=
		(1-4\beta \sigma_0^2)^{-1/2} 
		\mathrm{E} \Big[\exp\big\{\beta\sum_{i=1}^{|I|-2} \xi^2_{k_i} \big\} 
		e^{2\beta\xi^2_{k_{|I|-1}}}\Big] \\
		&\le \ldots \le
		(1-4\beta \sigma_0^2)^{-|I|/2} = e^{B|I|}, \qquad \text{with} 
		\quad B=\log(1-4\beta \sigma_0^2)^{-1/2}.
	\end{align*}

	\paragraph{Relation to paper \cite{vanderPas&Szabo&vanderVaart:2017}.}
	When the present paper was under review, paper \cite{vanderPas&Szabo&vanderVaart:2017}
	on the same topic appeared (with discussion, see also our contribution 
	\cite{Belitser&Nurushev:2017} to this discussion). 
	The main result of \cite{vanderPas&Szabo&vanderVaart:2017} is the adaptivity 
	of the confidence set constructed by the Bayesian approach  over 
	the sparsity scale of nearly black vectors (introduced in Section \ref{implication}) 
	within a grand space $\ell_0[p_n]$ for some $p_n\to \infty$, $p_n=o(n)$ as $n\to \infty$, 
	under the EBR condition. 
	The EBR condition introduced in 
	\cite{vanderPas&Szabo&vanderVaart:2017} is essentially a version of 
	our EBR condition adopted to the sparsity scale within the grand space $\ell_0[p_n]$. 
	It is not difficult to see that, within the asymptotic framework $n\to \infty$ and restricting 
	the values of $\theta$ to some grand space $\ell_0[p_n]$ with $p_n=o(n)$, 
	the EBR condition introduced in \cite{vanderPas&Szabo&vanderVaart:2017} is actually equivalent to 
	our EBR condition specified to that embedded sparsity scale with 
	appropriate choices of the constants involved.
	
	Restricting the values of $\theta$ to some grand space $\ell_0[p_n]$ excludes some 
	``almost sparse'' parameters  that are formally non-sparse (with many very small, 
	but nonzero, entries), but this is in fact necessary to ensure the asymptotic regime $n\to\infty$
	considered in \cite{vanderPas&Szabo&vanderVaart:2017}.
	The main differences of our approach and that of \cite{vanderPas&Szabo&vanderVaart:2017} are 
	the following. We obtain local results without relating to any sparsity scale,
	e.g., the true parameter $\theta$ may be not $\ell_0[p_n]$-sparse at all.
	For example, as a consequence we derive the results not only for $\ell_0[p_n]$,
	but also for other sparsity scales, such as \emph{weak} $\ell_s$-\emph
	{balls} $m_s[p_n]$.
	Next, we allow the error vector $\xi$ to be non-normal and even not necessarily 
	independent (but just satisfying Condition (A1)). Some of our constants in the proofs and conditions 
	may be more conservative, which is not surprising since we pursue a more general situation. 
	Finally, we derive non-asymptotic exponential concentration
	bounds, which give a refined characterization of the quality of coverage and size relation
	results (finer, than, e.g., Theorem $5$ from \cite{vanderPas&Szabo&vanderVaart:2017}, which is asymptotic
	in $n\to\infty$) and allow subtle analysis for various asymptotic regimes.
	
	We should mention that the derivation of our somewhat stronger results relies
	on certain explicit posterior expressions resulting from our choice of prior
	(mixture of normals, although the model is not assumed to be normal),
	whereas the horseshoe prior studied in \cite{vanderPas&Szabo&vanderVaart:2017} 
	leads to only implicit posterior quantities so that the authors had to overcome difficult  
	technical issues in the proofs.

	\subsection{The EBR condition}
	\label{subsec_ebr}
	\paragraph{A new perspective on EBR -- the EBR scale.}
	As mentioned in the introduction, it is impossible 
	to construct optimal (fully) adaptive confidence set in the minimax sense 
	over traditional smoothness and sparsity scales with a prescribed high coverage probability.  
	Namely, there exist ``deceptive'' parameters $\theta \in 
	\Theta'_0=\mathbb{R}^n \backslash \Theta_0$ for which the coverage property in \eqref{defconfball}
	may not hold for arbitrarily small $\alpha_1$. 
	Removing deceptive parameters $\Theta'_0$ and restricting to 
	the remaining set $\Theta_0$ of non-deceptive parameters resolves this issue. 
	This was the original motivation of introducing the EBR condition.

	An interesting additional feature of the EBR condition is that it leads to 
	\emph{slicing of the entire parameter space} $\mathbb{R}^n$. This opens up a new 
	perspective on the EBR and its role in the deceptiveness issue, which we explain next. 
	

	Note that the EBR condition $\theta \in \Theta_{\mathrm{eb}}(t,\tau)$ (see \eqref{cond_ebr}) is actually a family of 
	embedded conditions parametrized by $t\ge 0$: $\Theta_{\mathrm{eb}}(t_1,\tau) 
	\subseteq \Theta_{\mathrm{eb}}(t_2,\tau)$ for $t_1\le t_2$ and any $\tau\ge 0$.
	Note that, by the oracle definition, 
	$\Theta_{\mathrm{eb}}(\tau n,\tau) = \mathbb{R}^n$ for any $\tau>0$.
	An important observation is that this family of conditions effectively introduces a new scale
	$\cup_{t\ge 0} \Theta_{\mathrm{eb}}(t,\tau)=\cup_{0\le t\le \tau n} \Theta_{\mathrm{eb}}(t,\tau)$  (for any fixed $\tau>0$), to be called 
	the \emph{EBR scale}, with the structural  parameter $t\ge 0$ measuring 
	the allowed amount of deceptiveness for parameters $\theta\in \Theta_{\mathrm{eb}}(t,\tau)$. 
	Indeed, this scale  ``slices'' $\mathbb{R}^n$ in the sense that $\mathbb{R}^n=
	\cup_{0\le t\le \tau n} \Theta_{\mathrm{eb}}(t,\tau)$. The main benefit of introducing the EBR scale 
	is that it gives the slicing of the entire space that is very suitable for uncertainty quantification.
	Indeed, the dictum ``removing deceptive parameters'' becomes a very natural notion in terms of the scale $\cup_{0\le t\le \tau n}\Theta_{\mathrm{eb}}(t,\tau)$ 
	as it is nothing else but restricting the amount of deceptiveness $t$. This provides a new perspective 
	at the above mentioned ``deceptiveness'' issue: basically, each parameter $\theta\in\mathbb{R}^n$ has 
	a certain amount of ``deceptiveness''  that is measured by the excessive bias ratio $b_\tau(\theta)$ defined 
	by \eqref{def_b}, or the smallest $t$ for which $\theta\in  \Theta_{\mathrm{eb}}(t,\tau)$. The larger $t$, 
	the more deceptive parameters are allowed in $\Theta_{\mathrm{eb}}(t,\tau)$.
	A mild and controllable price for the uniformity over $\Theta_{\mathrm{eb}}(t,\tau)$  in the coverage 
	relation is the amount of inflating of the confidence ball needed  to provide a guaranteed high coverage
	for the parameters of deceptiveness at most $t$. 
	We should mention that the EBR scale is intrinsically tied to our Bayesian procedure 
	as it depends on the proposed family of priors $\{\pi_I, I\in\mathcal{I}\}$. 
	A different family may lead to a different EBR scale.
	Note however that a version of our EBR condition (adapted to the sparsity scale) 
	is still used for a different (horseshoe) prior in paper \cite{vanderPas&Szabo&vanderVaart:2017} .   
	
	Interestingly, slicing is also possible by the parameter $\tau>0$: $\mathbb{R}^n =\cup_{\tau\ge 0} \Theta_{\mathrm{eb}}(t,\tau)$ (for any $t>0$), the embedding goes in the opposite direction: the smaller the $\tau$, the weaker the EBR. 
	Namely, $\Theta_{\mathrm{eb}}(t,\tau_2)\subseteq \Theta_{\mathrm{eb}}(t,\tau_1)$ for 
	any $0\le \tau_1\le \tau_2$, $t>0$,  and the ``limiting'' EBR set $\lim_{\tau \downarrow 0} 
	\Theta_{\mathrm{eb}}(t,\tau)$ expands to the entire space: $\Theta_{\mathrm{eb}}(t,0)=\mathbb{R}^n$. 
	Besides, notice that the inflating factor in the confidence ball from Theorem \ref{th8} will not not increase 
	as $\tau \downarrow 0$ (in fact, it will decrease). A paradox seems to have emerged:  
	by considering very small $\tau$'s, we can have less deceptiveness without any price in the coverage relation. However, this paradox is resolved by reminding that the coverage relation from 
	Theorem \ref{th8} does not hold for arbitrarily small $\tau_0$ because in \eqref{def_I_*}
	$\tau_0>\tfrac{1+\varrho}{1-\varrho} \bar{\tau}$, showing that ``there is no free lunch''.  
	The lower bound for $\tau_0$ can be relaxed (made smaller) for specific distribution of $\xi$, 
	but it will always be some positive threshold reflecting the complexity of $\xi$.

	\paragraph{The EBR does not affect the minimaxity over the sparsity scale $\ell_0[p]$.}
	The EBR condition is mild from the minimax point of view in the following sense: 
	if we take the traditional sparsity  class $\ell_0[p] = \{\theta \in \mathbb{R}^n: |I^*(\theta)|
	=\|\theta\|_0 \le p\}$ for $p \in \mathbb{N}_n$ and remove non-EBR parameters, then
	the minimax rate over the remaining part will not change (up to a constant).  We outline the argument below. 
	The minimax estimation rate was established by Birg\'e and Massart \cite{Birge&Massart:2001} 
	(Theorem 4 from \cite{Birge&Massart:2001}, formulated in our notation): for some universal constant $c>0$,
	\[
	r^2(\ell_0[p]) \triangleq \inf_{\hat{\theta}} \sup_{\theta \in \ell_0[p]} \mathrm{E}_{\theta} \|\hat{\theta}-\theta\|^2  
	\ge c \sigma^2p \log(en/p).  
	\]
	The proof is based on considering the subset $\mathcal{B}_1(p)
	= \{ \theta \in \mathbb{R}^n: |I^*(\theta)|\le p, \, |\theta_i| \le \sigma^2\log(en/p)\} \subset \ell_0[p]$ 
	and establishing the required lower bound for the minimax risk $R(\mathcal{B}_1(p))$ 
	over the set $\mathcal{B}_1(p)$, thus obtaining $r^2(\ell_0[p]) \ge r^2(\mathcal{B}_1(p))\ge c \sigma^2p\log(en/p)$. 
	Inspecting all the steps in the proof, we see that essentially the same lower bound 
	(with a different constant $c$) holds for another subset of $\ell_0[p]$: 
	$\mathcal{B}_2(p,\tau_0)= \{ \theta \in \mathbb{R}^n: |I^*(\theta)|
	=p, \, 2\tau_0 \sigma^2\log(en/p)\le |\theta_i| \le (2\tau_0+1)\sigma^2\log(en/p)\;
	\mbox{for all}\; i \in I^*(\theta)\}$, for $\tau_0$ defined in \eqref{def_I_*}.
	For each $\theta \in \mathcal{B}_2(p,\tau_0)$, we have $I_*(\theta)=I_o^{\tau_0}(\theta) = I^*(\theta)$ 
	so that $I_*^c(\theta)=\varnothing$, $|I_*(\theta)|=p$, and the EBR condition is trivially satisfied for any $t\ge 0$:
	\[
	\frac{\sum_{i\in I_*^c} \theta_i^2}{\sigma^2(1+|I_*|\log(en/|I_*|))}=0\le t .
	\] 
	This means that $r^2(\ell_0[p]\cap \Theta_{\mathrm{eb}}(t)) 
	\ge r^2(\mathcal{B}_2(p,\tau_0))\ge c \sigma^2p\log(en/p)$.

	\section{Simulations}
	\label{simulations}
	
	Here we present a small simulation study according to the model \eqref{model} with 
	$\xi_i\overset{\rm ind}{\sim}\mathrm{N}(0,1)$, $\sigma=1$ and $n=500$. 
	We used signals $\theta=(\theta_{1},\ldots,\theta_{n})$ 
	of the form $\theta=(0,\ldots,0, A,\ldots,A)$, where $p=\#(\theta_{0,i}\neq 0)$ 
	last coordinates of $\theta$ are equal  to a fixed number $A$. 
	Different sparsity levels $p_n\in\{25,50,100\}$ and ``signal strengths'' 
	$A\in\{3,4,5\}$ are considered.  In case $\xi_i \overset{\rm ind}{\sim} \mathrm{N}(0,1)$, 
	the Conditions \eqref{cond_nonnormal} and \eqref{cond_technical} are fulfilled with 
	$\beta=0.4, B=1$ and $\varkappa>3.24$. 
	The idea is to construct an empirical counterpart of the ball $B\big(\hat{\theta},[b(\theta)+1)M_2\hat{r}^2]^{1/2}\big)$ which appears in Remark \ref{rem4} (for $M=0$).
	We consider $\hat{I}$ and $\check{\theta}$ 
	defined respectively by \eqref{I_MAP} and \eqref{emp_emp_posterior}, 
	and take $\hat{b}=
	\frac{\sum_{i\not\in \hat{I}} (X_i^2-1)}{|\hat{I}|\log(en/|\hat{I}|)}$ as the estimator 
	of $b(\theta)$ defined by \eqref{def_b}. 
	
	
	When computing $\hat{I}$ given by \eqref{I_MAP}, an important choice is that 
	of parameter $\varkappa>0$.
	In our simulation study, we choose $\varkappa$ via a \emph{cross-validation} procedure. 
	For that, we create two independent  normal  samples $X'_i=X_i+\eta_i$ and  $X''_i=X_i-\eta_i$, 
	where simulated independent standard normal $\eta_i$'s are assumed to be 
	independent of $\xi_i$, $i\in[n]$. 
	Then $X'_i$ and $X''_i$ are independent random variables with
	means $\theta_i$ and variances $2$. In words, the observation sample can be duplicated at the cost
	of multiplying the variance by $2$. Using these samples, we estimate $\varkappa>0$ as follows: let $\check{\theta}'=\check{\theta}'(\hat{I}')=(X'_i\mathrm{1}\{i\in \hat{I}'\}, i\in[n])$, where 
	$\hat{I}'(\varkappa)=\arg\!\min_{I\in\mathcal{I}}\Big\{-\sum_{i \in I}(X_i')^2 +2(2\varkappa +1)
	|I| \log\big(\tfrac{en}{|I|}\big)\Big\}$, then $\hat{\varkappa}
	=\arg\!\min_{\varkappa\in (0,\log n]}\|\hat{I}'(\varkappa)-X''\|^2$. 
	Now, let $\hat{I}=\hat{I}(\hat{\varkappa})$ be defined by \eqref{I_MAP} 
	with $\varkappa=\hat{\varkappa}$. 
	Finally, consider the confidence ball 
	$B\big(\check{\theta}, [\check{M}(\hat{b}+1)\hat{r}^2]^{1/2}\big)$  
	around $\check{\theta}$ defined by \eqref{emp_emp_posterior}
	where  $\hat{b}=\frac{\sum_{i\in \hat{I}^c} (X_i^2-1)}{|\hat{I}|\log(en/|\hat{I}|)}$ 
	and  $\hat{r}^2=|\hat{I}|\log(en/|\hat{I}|)$ given by \eqref{check_r}. 
	Recall that this confidence ball construction is inspired by local result formulation from Remark \ref{rem4}.
	The multiplicative factor $\check{M}$ is intended to trade-off the size of the ball 
	against its coverage probability. 
	We take $\check{M}=2$, an intuitive justification for this choice would be the fact that  
	we doubled the variance in the randomization procedure when duplicating the sample.
	For each sparsity level $p\in\{25,50,100\}$ and signal strength
	$A\in\{3,4,5\}$, we simulated $100$ data vectors $X$ of dimension $n=500$ 
	from the model \eqref{model} and computed the average squared radius 
	by $\overline{\hat{R}^2}$. Table $\ref{table:2}$ shows the ratio of $\overline{\hat{R}^2}$ 
	to the oracle  radial rate $r^2(\theta)=|I_o|\log (en/|I_o|)=p\log (en/p)$, 
	where the oracle $I_o$  is defined by \eqref{oracle}, and the frequency $\bar{\alpha}$ 
	of the event that confidence ball $B\big(\check{\theta},[\check{M}(\hat{b}+1)\hat{r}^2]^{1/2}\big)$ 
	contains the signal $\theta$, respectively. 
	The former quantity estimates the average inflating factor with respect to the oracle rate, 
	and the latter estimates the coverage of the constructed confidence ball.
	We see that the most difficult case is $p=25$ and $A=3$. An informal interpretation of the simulation results for this case: we need to blow up the confidence ball of the 
	oracle rate radius approximately by the factor $3.24$ in order to cover 
	a few 
	small 
	needles in a haystack with coverage probability approximately $0.92$.
	

	\begin{table*}[ht]
		\centering
		\caption[l]{The ratio $\overline{\hat{R}^2}/r^2(\theta)$
			and the frequency $\bar{\alpha}$ of the event that the confidence 
			ball $B\big(\check{\theta},[\check{M}(\hat{b}+1)\hat{r}^2]^{1/2}\big)$ contains the  signal 
			$\theta=(0,\ldots,0,A,\ldots,A)$ (where $p$ last coordinates are equal to $A$) 
			computed for $100$ vectors $X$ simulated from 
			\eqref{model} with $n=500$, $\sigma=1$.}
		\label{table:2}
		\begin{tabular}{ccccrcccrccc} 
			\hline
			\addlinespace[0.5ex]
			$p$ &  & $\mathbf{25}$ &   & & & $\mathbf{50}$ &  & &  & $\mathbf{100}$ &  \\
			\addlinespace[-0.5ex]
			\cmidrule{2-4} \cmidrule{6-8} \cmidrule{10-12}
			$A$ & $\mathbf{3}$ & $\mathbf{4}$ & $\mathbf{5}$ & & $\mathbf{3}$ & $\mathbf{4}$ 
			& $\mathbf{5}$ & & $\mathbf{3}$ & $\mathbf{4}$ & $\mathbf{5}$ \\
			\hline
			\addlinespace[0.5ex]
			$\overline{\hat{R}^2}/r^2(\theta)$ & 3.24 &2.73 & 2.21 & 
			& 3.24 & 2.14 & 2.07 & & 2.26 & 1.83 & 1.85\\
			$\bar{\alpha}$& $0.92$ & $0.97$  & $0.98$  & & $0.99$  & $0.99$  & $1$  & 
			& $0.96$ & $1$  & $1$ \\
			\hline
		\end{tabular}
	\end{table*}

	\section{Technical lemmas}
	\label{proofs_lemmas}
	First we provide a couple of technical lemmas used in the proofs of the main results.
	\begin{remark} 
		\label{rem1a}
		Notice that in the below lemma we established the same bound for the both quantities
		$\mathrm{E}_{\theta}\hat{\pi}(I|X)=\mathrm{E}_{\theta}\tilde{\pi}(I|X)$ and 
		$\mathrm{E}_{\theta}\mathbbm{1}\{\hat{I}=I\}=\mathrm{P}_{\theta}(\hat{I}= I)$.
		The proofs of the properties of $\check{\pi}(\vartheta|X)$ and $\check{\theta}$ 
		are exactly in the same as for $\tilde{\pi}(\vartheta|X)$ 
		and $\tilde{\theta}$, with the only difference that 
		everywhere (in the claims and in the proofs)
		$\hat{\pi}(I \in \mathcal{G}|X)$ should be read as 
		$\tilde{\pi}(I \in \mathcal{G}|X)$ in case $\hat{\pi}=\tilde{\pi}$; 
		and as $\mathbbm{1}\{\hat{I}\in\mathcal{G}\}$ in case $\hat{\pi}=\check{\pi}$, 
		for all $\mathcal{G}\subseteq \mathcal{I}$ that appear in the proof.
		Hence, $\mathrm{E}_{\theta}\hat{\pi}(I\in\mathcal{G}|X)
		=\mathrm{E}_{\theta}\tilde{\pi}(I\in\mathcal{G}|X)$ in the former case,
		and $\mathrm{E}_{\theta}\hat{\pi}(I\in\mathcal{G}|X)
		=\mathrm{P}_{\theta}(\hat{I}\in\mathcal{G})$ in the latter case.
	\end{remark} 
	
	\begin{lemma}
		\label{lemma1}
		Let Condition \eqref{cond_nonnormal} be fulfilled. 
		Then for any $\theta\in\mathbb{R}^n$ and any $I,I_0\in\mathcal{I}$,  
		\begin{align*}
			\mathrm{E}_{\theta}\hat{\pi}(I|X)
			\le\big[\tfrac{\lambda_{I}}{\lambda_{I_0}}\big]^h
			\exp\Big\{B_h\!\!\sum_{i\in I\backslash I_0}\!\!
			\tfrac{\theta_i^2}{\sigma^2} -A_h\!\!\sum_{i\in I_0\backslash I}\!\!
			\tfrac{\theta_i^2}{\sigma^2}+C_h|I_0| \log(\tfrac{en}{|I_0|}) -D_h|I| \log(\tfrac{en}{|I|})\Big\},
		\end{align*}
		where
		$h =\tfrac{2\beta}{3}$, 
		$A_h=\tfrac{\beta}{6}$, $B_h=\tfrac{2\beta}{3}$, $C_h=\tfrac{\beta+B}{3}$ and 
		$D_h =\tfrac{\beta-2B}{3}$. If $I\backslash I_0 = \varnothing$, 
		the bound holds also for $h=\beta$ with $A_h=\frac{\beta}{3}$, $B_h=0$,
		$C_h=\frac{\beta}{2}+B$, $D_h=\frac{\beta}{2}$.
		If $I_0\backslash I=\varnothing$, the bound holds also for $h=\beta$ with 
		$A_h=0$, $B_h=\beta$, $C_h=\tfrac{\beta}{2}$, $D_h=\tfrac{\beta}{2}-B$.
	\end{lemma} 
	\begin{proof}[Proof of Lemma \ref{lemma1}]
		Recall that $\mathrm{P}_{X,I} =\phi\big(X_i\mathbbm{1}\{i\not \in I\},0,\sigma^2+
		K_n(I) \sigma^2\mathbbm{1}\{i\in I\}\big)$.
		In case $\hat{\pi}(I|X)=\tilde{\pi}(I|X)$, we get by \eqref{emp_P(I|X)} 
		that, for any $I,I_0\in\mathcal{I}$ and any $h\in [0,1]$,
		\begin{align}
			\mathrm{E}_{\theta}\hat{\pi}(I|X)&=
			\mathrm{E}_{\theta}\tilde{\pi}(I|X)
			=
			\mathrm{E}_{\theta}\frac{\lambda_I \mathrm{P}_{X,I}}
			{\sum_{J\in\mathcal{I}} \lambda_J\mathrm{P}_{X,J}} 
			\le 
			\mathrm{E}_{\theta}\Big(\frac{\lambda_I \mathrm{P}_{X,I}}
			{\lambda_{I_0}\mathrm{P}_{X,I_0}}\Big)^h
			\label{relation_P2} \\  
			\label{relation_P1}
			&= \mathrm{E}_{\theta}\Big[\frac{\lambda_I \prod_{i=1}^n 
				\phi\big(X_i\mathbbm{1}\{i\not \in I\},0,\sigma^2+
				K_n(I) \sigma^2\mathbbm{1}\{i\in I\}\big)}{\lambda_{I_0} \prod_{i=1}^n 
				\phi\big(X_i\mathbbm{1}\{i\not \in I_0\},0,\sigma^2+
				K_n(I_0)\sigma^2\mathbbm{1}\{i\in I_0\}\big)}\Big]^h \notag\\
			&=\big[\tfrac{\lambda_{I}}{\lambda_{I_0}}\big]^h  
			\mathrm{E}_{\theta} \exp\Big\{\tfrac{h}{2}\Big[
			\sum_{i\in I \backslash I_0}\tfrac{X_i^2}{\sigma^2}
			-\sum_{i\in   I_0 \backslash I}\tfrac{X_i^2}{\sigma^2}
			+|I_0| \log(\tfrac{en}{|I_0|})-|I|\log(\tfrac{en}{|I|})\Big]\Big\}.
		\end{align}
		In case $\hat{\pi}(I|X)=\mathbbm{1}\{\hat{I}=I\}$, 
		by the definition \eqref{I_MAP} of $\hat{I}$ 
		and the Markov inequality, we derive that, for any $I,I_0\in\mathcal{I}$ 
		and any $h\ge 0$
		\begin{align*}
			\mathrm{E}_{\theta}\hat{\pi}(I|X) 
			=\mathrm{P}_{\theta}(\hat{I}= I)\le\mathrm{P}_{\theta}\Big(\frac{\tilde{\pi}(I|X)}
			{\tilde{\pi}(I_0|X)}\ge 1\Big)\le 
			\mathrm{E}_{\theta}\Big[\frac{\tilde{\pi}(I|X)}{\tilde{\pi}(I_0|X)}\Big]^h
			=\mathrm{E}_{\theta}\Big(\frac{\lambda_I \mathrm{P}_{X,I}}
			{\lambda_{I_0}\mathrm{P}_{X,I_0}}\Big)^h,
		\end{align*}
		which yields exactly the bound \eqref{relation_P2}, and hence the bound 
		\eqref{relation_P1} again.
		
		Using H\"older's inequality, Condition  \eqref{cond_nonnormal} and the two elementary facts 
		$X_i^2 \le 2\theta_i^2 +2\sigma^2 \xi_i^2$  and $-X_i^2 \le -\tfrac{\theta_i^2}{2} +\sigma^2 \xi_i^2$ , 
		we obtain
		\begin{align*}
			\mathrm{E}_{\theta}\exp\Big\{\tfrac{\beta}{3}\Big[ \sum_{i\in I\backslash I_0} \tfrac{X_i^2}{\sigma^2} 
			-\sum_{i\in I_0\backslash I} \tfrac{X_i^2}{\sigma^2} \Big]  \Big\} 
			&\le 
			\Big(\mathrm{E}_{\theta} e^{\tfrac{\beta}{2} \sum_{i\in I\backslash I_0} \tfrac{X_i^2}{\sigma^2}}\Big)^{2/3}
			\Big(\mathrm{E}_{\theta} e^{-\beta\sum_{i\in I_0\backslash I} \tfrac{X_i^2}{\sigma^2}} \Big)^{1/3}\\
			&\le 
			\exp\Big\{\tfrac{2\beta}{3} \sum_{i\in I\backslash I_0} \tfrac{\theta_i^2}{\sigma^2} + \tfrac{2B}{3}|I\backslash I_0|
			-\tfrac{\beta}{6}\sum_{i\in I_0\backslash I}\tfrac{\theta_i^2}{\sigma^2} +\tfrac{B}{3}|I_0\backslash I|\Big\}.
		\end{align*}
		Since $|I\backslash I_0| \le|I|\le  |I|\log(\tfrac{en}{|I|})$ and 
		$|I_0\backslash I| \le|I_0|\le  |I_0|\log(\tfrac{en}{|I_0|})$, 
		the lemma follows for $h=\frac{2\beta}{3}$ from the last display and \eqref{relation_P1}. 
		
		If  $I\backslash I_0 = \varnothing$, we take $h=\beta$ in \eqref{relation_P1} 
		and combine this with $\mathrm{E}_{\theta}\exp\big\{-\frac{\beta}{2} \sum_{i\in I_0\backslash I} 
		\tfrac{X_i^2}{\sigma^2} \big\} \le \exp\big\{-\frac{\beta}{3} \sum_{i\in I_0\backslash I} 
		\tfrac{\theta_i^2}{\sigma^2} +B |I_0\backslash I|\big\}$,
		which holds in view of Condition  \eqref{cond_nonnormal} and 
		$-\frac{X_i^2}{\sigma^2} \le -\frac{2\theta_i^2}{3\sigma^2} +2\xi_i^2$, 
		as $(a+b)^2\ge 2a^2/3-2b^2$.
		If $I_0\backslash I = \varnothing$, we take $h=\beta$ in \eqref{relation_P1} 
		and combine this  with $\mathrm{E}_{\theta}\exp\big\{\frac{\beta}{2} 
		\sum_{i\in I\backslash I_0}\tfrac{X_i^2}{\sigma^2} \big\}\le\exp\big\{\beta\sum_{i\in I\backslash I_0} 
		\tfrac{\theta_i^2}{\sigma^2} +B |I\backslash I_0|\big\}$ which holds 
		in view of Condition  \eqref{cond_nonnormal} and $\frac{X_i^2}{\sigma^2} \le 
		\frac{2\theta_i^2}{\sigma^2} +2\xi_i^2$.
	\end{proof}

	Note that above lemma holds for any set $I_0 \in\mathcal{I}$. 
	By taking $I_0=I_o$ defined by \eqref{oracle}, we obtain the following lemma.
	\begin{lemma}
		\label{lemma2}
		Let Conditions  \eqref{cond_nonnormal} and \eqref{cond_technical} be fulfilled. 
		Then there exist positive constants 
		$c_1=c_1(\varkappa)>2$, $c_2$ and $c_3=c_3(\varkappa)$ such that
		for any $\theta\in\mathbb{R}^n$
		\begin{align*}
			\mathrm{E}_{\theta}\hat{\pi}(I |X) \le \big(\tfrac{ne}{|I|}\big)^{-c_1 |I|}  
			\exp\big\{-c_2\sigma^{-2}\big[r^2(I,\theta)-c_3r^2(\theta)\big]\big\}. 
		\end{align*}
	\end{lemma}
	\begin{proof}[Proof of Lemma \ref{lemma2}]
		With constants $h$, $A_h,B_h,C_h,D_h$ given in Lemma \ref{lemma1}, define 
		the constant $c_1=c_1(\varkappa)=\varkappa h +D_{h}-A_{h} =\tfrac{2\beta\varkappa}{3}+\tfrac{\beta-2B}{3}
		-\frac{\beta}{6} >2$ as $\varkappa>\bar{\varkappa}$ by Condition (A2).  
		Since $\varkappa h+D_{h}=c_1+A_{h}$, the definition \eqref{prior_lambda} 
		of  $\lambda_I$ entails that
		\begin{align*}
			&\big(\tfrac{\lambda_{I}}{\lambda_{I_0}}\big)^{h}
			\exp\big\{C_{h}|I_0| \log(\tfrac{en}{|I_0|})-D_{h}|I| \log(\tfrac{en}{|I|})\big\} \notag\\&
			=(\tfrac{ne}{|I|})^{-c_1 |I|} \exp\big\{(\varkappa h +C_h)|I_0|\log(\tfrac{en}{|I_0|}) - 
			A_h|I| \log(\tfrac{en}{|I|})\big\}.
		\end{align*}
		Using the last relation and Lemma $\ref{lemma1}$ with $I_0=I_o$, 
		we  bound   
		\begin{align*}
			&\mathrm{E}_{\theta}\hat{\pi}(I|X)\le\big[\tfrac{\lambda_{I}}{\lambda_{I_o}}\big]^{h}
			\exp\Big\{B_{h}\sum_{i\in I \backslash I_o} \tfrac{\theta_i^2}{\sigma^2}
			-A_{h}\sum_{i\in I_o \backslash I}\tfrac{\theta_i^2}{\sigma^2}
			+C_{h}|I_o| \log(\tfrac{en}{|I_o|})-D_{h}|I| \log(\tfrac{en}{|I|})\Big\}\\
			&= (\tfrac{ne}{|I|})^{-c_1 |I|} \exp\Big\{\!-\!A_{h}\!\!\!
			\sum_{i\in I_o \backslash I}\tfrac{\theta_i^2}{\sigma^2} -A_{h}|I| \log(\tfrac{en}{|I|})+
			B_{h}\sum_{i\in I \backslash I_o}\!\!\tfrac{\theta_i^2}{\sigma^2}
			+(\varkappa h +C_h)|I_o| \log(\tfrac{en}{|I_o|})\Big\}.
		\end{align*}
		The claim of the lemma follows with the constants  $c_1= (4\beta\varkappa+\beta-4B)/6 >2$, 
		$c_2=A_h=\beta/6$ and $c_3=c_3(\varkappa)=\max\{B_h, \varkappa h +C_h\}/A_h = 
		(\varkappa h +C_h)/A_h = 4\varkappa + 2(\beta+B)/\beta$.
	\end{proof}

	\begin{lemma}
		\label{lemma4}
		Let $Y_1, \ldots, Y_n$ be  some 
		random variables such that, for any $I\in \mathcal{I}$,
		$\mathrm{E}e^{t \sum_{i\in I} Y_i} \le A_{|I|}(t) $ for some $t>0$ and $A_k(t)$. 
		Let $Y_{[1]} \ge Y_{[2]}\ge\ldots \ge Y_{[n]}$.
		Then, for any $k\in\mathbb{N}_n$ and $C,c\ge 0$, 
		\begin{align*}
			& \mathrm{P}\Big( \sum_{i=1}^k Y_{[i]} \ge Ck \log(\tfrac{en}{k})+c\Big) 
			\le A_k(t) \exp\{-(Ct-1)k \log(\tfrac{en}{k})-ct\}, \\
			&\mathrm{E} \sum_{i=1}^k Y_{[i]} \le t^{-1} \big[k \log(\tfrac{en}{k})+\log (A_k(t))\big].
		\end{align*}
		
		In particular, if $\xi_1, \ldots, \xi_n \overset{\rm ind}{\sim} \mathrm{N}(0,1)$,
		then for any $k\in\mathbb{N}_n$, $C,c\ge 0$ 
		\[
		\mathrm{P}\Big( \sum_{i=1}^k \xi_{[i]}^2 \ge Ck \log\big(\tfrac{en}{k}\big)+c\Big) 
		\le \big(\tfrac{en}{k}\big)^{-(0.4C-2)k} e^{-0.4c}, \quad
		\mathrm{E} \sum_{i=1}^k \xi_{[i]}^2 \le 6k\log\big(\tfrac{en}{k}\big).
		\]
	\end{lemma}
	\begin{proof}
		By Jensen's inequality, we derive  
		\begin{align*}
			&\exp\Big\{t\mathrm{E}\sum_{i=1}^k Y_{[i]}\Big\}
			\le\mathrm{E}\exp\Big\{t\sum_{i=1}^k Y_{[i]} \Big\}
			\le \sum_{I: |I|=k}\mathrm{E}\exp\Big\{t\sum_{i\in I} Y_i\Big\}
			\le \tbinom{n}{k} A_k(t).
		\end{align*}
		Then $\mathrm{E}\exp\big\{t\sum_{i=1}^k Y_{[i]} \big\}\le \tbinom{n}{k} A_k(t)
		\le e^{k\log(\tfrac{en}{k})+\log(A_k(t))}$,
		where we used $\binom{n}{k} \le (\tfrac{en}{k})^k$.
		This and the (exponential) Markov inequality yield the first relation:
		\begin{align*}
			\mathrm{P}\Big( \sum_{i=1}^k Y_{[i]}\ge Ck \log\big(\frac{en}{k}\big)+c\Big) 
			\le A_k(t) \exp\{-(Ct-1)k \log(\tfrac{en}{k})-ct\}.  
		\end{align*}
		The first display  implies also the second relation: $\mathrm{E} \sum_{i=1}^k Y_{[i]}
		\le t^{-1}[\log \binom{n}{k}+\log(A_k(t))]$.
		
		As to the normal case, for any $I \in \mathcal{I}$ and any $t < \frac{1}{2}$ 
		we have that $\mathrm{E}\exp\big\{ t\sum_{i\in I}\xi_i^2\big\}=
		(1-2t)^{-|I|/2}=A_{|I|}(t)$. Since $A_k(t) \le e^k \le e^{k\log(\tfrac{en}{k})}$ for any $t \le(1-e^{-2})/2<0.43$,
		the first assertion for the normal case follows by taking $t=0.4$.
		By taking $t=\tfrac{1}{4}$, the second assertion follows since
		$\mathrm{E}\sum_{i=1}^k \xi^2_{[i]}
		\le 4k\log(\tfrac{en}{k})+2k\log 2 \le 6k\log(\tfrac{en}{k})$.
	\end{proof}
	
	This lemma is useful if $A_k(t) \le C_1(\tfrac{en}{k})^{C_2 k}$ for some $t,C_1,C_2>0$;
	in particular, for $Y_i=\xi^2_i$, where the $\xi_i$'s satisfy Condition \eqref{cond_nonnormal}.
	Then Lemma \ref{lemma4} applies with $t=\beta$ and $A_k(\beta)= e^{Bk}$: 
	\begin{align}
		\label{lem4_rel}
		\mathrm{P}\Big(\sum_{i=1}^{k} \xi_{[i]}^2 \ge \tfrac{(1+B)}{\beta}k \log(\tfrac{en}{k})+ M \Big) 
		\le \exp\{-\beta M\}, \quad k \in \mathbb{N}_n, \; M \ge 0.
	\end{align}
	
	
	\section{Proofs of the theorems}
	\label{proofs_theorems}
	Here we gather the proofs of the theorems. By $C_0, C_1, C_2$ etc., 
	denote constants which are different in different proofs.
	Recall that $Y_{[1]} \ge Y_{[2]}\ge\ldots \ge Y_{[n]}$ denote the ordered $Y_1,\ldots, Y_n$.

	\begin{proof}[Proof of Theorem \ref{th1}] 
		Recall the constants $c_1, c_2, c_3$ defined in the proof of Lemma 
		\ref{lemma2}. Let $M_0=2c_3(6+\tfrac{1+B}{\beta})$. Introduce the subfamily of index sets 
		$\mathcal{S}_M=\mathcal{S}_M(\theta)=\{ I \in \mathcal{I}: r^2(I,\theta) 
		\le c_3 r^2(\theta)+\tfrac{\beta}{40(1+B)}M\sigma^2\}$,  
		$m=m_M(\theta)=\max\{|I|: I \in \mathcal{S}_M\}$,
		and the event $A_M=A(\theta)= \big\{ \sum_{i=1}^m \xi_{[i]}^2 
		\le \tfrac{(1+B)}{\beta}m \log(\tfrac{en}{m}) +\frac{M}{8}\big\}$.  
		We have
		\begin{align*}
			&\hat{\pi}\big(\|\vartheta-\theta\|^2\ge M_0r^2(\theta)+ M\sigma^2 | X\big)
			\le \mathbbm{1}_{A^c_M} +\hat{\pi}(I \in \mathcal{S}_M^c|X) \\
			&\quad + 
			\sum_{I\in\mathcal{S}_M} \mathbbm{1}_{A_M}
			\hat{\pi}_I\big(\|\vartheta-\theta\|^2 
			\ge M_0r^2(\theta)+ M\sigma^2| X\big)\hat{\pi}(I|X)
			=T_1+T_2+T_3.
		\end{align*}
		Now we bound the quantities $\mathrm{E}_{\theta}T_1$, $\mathrm{E}_{\theta}T_2$ 
		and $\mathrm{E}_{\theta}T_3$.
		
		First, we bound $\mathrm{E}_{\theta}T_1$ by using  Lemma \ref{lemma4} (see also \eqref{lem4_rel}):
		\begin{align}
			\mathrm{E}_{\theta} T_1=\mathrm{P}_{\theta}(A^c_M) &= 
			\mathrm{P}\Big(\sum_{i=1}^{m} \xi_{[i]}^2 >  \tfrac{(1+B)}{\beta}m \log(\tfrac{en}{m}) +\tfrac{M}{8}\Big)  
			\le \exp\big\{-\beta M/8\big\}.
			\label{th1_T1}
		\end{align}
		
		Let us bound $\mathrm{E}_{\theta} T_2$. 
		Since $\tbinom{n}{k} \le (\frac{en}{k})^k$ and  $c_1>2$,  the following relation holds:
		\begin{align}
			\label{th1_fact1}
			\sum_{I\in\mathcal{I}}\big(\tfrac{ne}{|I|}\big)^{-c_1 |I|}=
			\sum_{k=0}^n \tbinom{n}{k} \big(\tfrac{en}{k}\big)^{-c_1k} 
			\le \sum_{k=0}^n\big(\tfrac{en}{k}\big)^{-k(c_1-1)}
			\le (1-e^{1-c_1})^{-1}\triangleq C_0.
		\end{align} 
		If $I \in \mathcal{S}^c_M$, then $r^2(I,\theta)> c_3 r^2(\theta)+\tfrac{\beta}{40(1+B)}M\sigma^2$.
		Using this, Lemma \ref{lemma2}  and \eqref{th1_fact1}, we  bound $\mathrm{E}_{\theta} T_2$:
		\begin{align}
			\mathrm{E}_{\theta} T_2
			&=\sum_{I\in \mathcal{S}^c_M} \mathrm{E}_{\theta}\hat{\pi}(I|X)  
			\le \sum_{I\in\mathcal{S}^c_M}   \big(\tfrac{ne}{|I|}\big)^{-c_1|I|}
			\exp\big\{-c_2\sigma^{-2}
			\big[r^2(I,\theta)-c_3 r^2(\theta)\big]\big\} \notag\\
			\label{th1_T2}
			&\le \sum_{I\in\mathcal{I}} \big(\tfrac{ne}{|I|}\big)^{-c_1|I|}  
			\exp\{-c_2 \beta M/(40(1+B))\}\le C_0\exp\{-c_2\beta M/(40(1+B))\}  .
		\end{align}

		It remains to bound  $\mathrm{E}_{\theta} T_3$.
		For each $I \in \mathcal{S}_M$, $\sigma^2|I| \log(en/|I|) \le r^2(I,\theta)\le c_3 r^2(\theta)
		+\tfrac{\beta}{40(1+B)}M\sigma^2$. Since  $m=\max\{|I|: I \in \mathcal{S}_M\}$, then
		$\sigma^2m\log(\tfrac{en}{m}) \le c_3 r^2(\theta)+\tfrac{\beta}{40(1+B)}M \sigma^2$. Thus, 
		for any $I \in \mathcal{S}_M$, the event $A_M$ implies that
		$\sum_{i\in I} \xi_i^2\le\sum_{i=1}^{m} \xi_{[i]}^2 \le 
		\tfrac{(1+B)}{\beta} m \log(\tfrac{en}{m}) +\frac{M}{8}
		\le \tfrac{(1+B)}{\beta} c_3\sigma^{-2}r^2(\theta)+ \tfrac{3M}{20}$.  
		Denote for brevity $\Delta_M(\theta)=M_0r^2(\theta)+ M\sigma^2$
		and recall that $\sum_{i \in I^c}\theta_{i}^2\le r^2(I,\theta)\le c_3 r^2(\theta)
		+\tfrac{\beta}{40(1+B)}M\sigma^2\le c_3 r^2(\theta)+\tfrac{M}{40}\sigma^2$ 
		for any $I \in \mathcal{S}_M$. Then for any $I \in \mathcal{S}_M$
		\begin{align}
			\label{th7_rel1}
			A_M \subseteq 
			\Big\{\frac{\Delta_M(\theta)}{2}-\sigma^2 \sum_{i\in I} \xi_i^2 - \sum_{i \in I^c}\theta_i^2
			\ge\Big[\frac{M_0}{2}-\frac{1+B+\beta}{\beta}c_3\Big] r^2(\theta)+\frac{13M\sigma^2}{40}\Big\}.
		\end{align}
		According to \eqref{emp_poster_I}, 
		$\hat{\pi}_I(\vartheta|X) =\bigotimes_{i=1}^n\mathrm{N}\big(X_i(I),\sigma_i^2(I) \big)$,
		with $X_i(I)= X_i\mathbbm{1}\{i\in I\}$ and $\sigma_i^2(I) 
		=\tfrac{K_n(I)\sigma^2 \mathbbm{1}\{i\in I\}}{K_n(I)+1}$. 
		Let  $\mathrm{P}_Z$ be the measure of $Z=(Z_1, \ldots, Z_n)$, 
		with $Z_i\overset{\rm ind}{\sim}\mathrm{N}(0,1)$.
		By using \eqref{th7_rel1}, the fact that
		$\tfrac{r^2(\theta)}{\sigma^2} \ge c_3^{-1} (m\log(\tfrac{en}{m}) -\tfrac{\beta}{40(1+B)}M)$ and 
		Lemma \ref{lemma4} (now applied to the Gaussian case),  we obtain
		that, for any $I\in\mathcal{S}_M$,
		\begin{align*}
			\hat{\pi}_I \big (\|\vartheta &- \theta\|^2 
			\ge M_0r^2(\theta)+ M\sigma^2 | X\big)\mathbbm{1}_{A_M}
			= \mathrm{P}_Z\Big(\sum_{i=1}^n(\sigma_i(I) Z_i +X_i(I)-\theta_{i})^2 \ge \Delta_M(\theta)\Big)
			\mathbbm{1}_{A_M} \\
			&\le 
			\mathrm{P}_Z\Big(\sum_{i=1}^n \sigma_i^2(I) Z_i^2 \ge\tfrac{\Delta_M(\theta)}{2}  
			-\sum_{i=1}^n (X_i(I) - \theta_{i})^2\Big)\mathbbm{1}_{A_M}\\
			&\le
			\mathrm{P}_Z\Big(\sum_{i\in I} \sigma^2 Z_i^2 \ge
			\tfrac{\Delta_M(\theta)}{2}
			-\sum_{i\in I} \sigma^2\xi_i^2 - \sum_{i \in I^c}\theta_{i}^2\Big)\mathbbm{1}_{A_M}\\  
			&\le 
			\mathrm{P}_Z\Big(\sum_{i\in I} Z_i^2 \ge\big[\tfrac{M_0}{2}-\big(\tfrac{1+B}{\beta}+1\big)c_3\big]
			\frac{r^2(\theta)}{\sigma^2}+\tfrac{13M}{40}\Big)\\
			\label{th1_rel03}
			&\le 
			\mathrm{P}_Z\Big(\sum_{i=1}^m Z_{[i]}^2 \ge \big(\tfrac{M_0}{2c_3}-\tfrac{1+B}{\beta}-1\big)
			\big[m\log(\tfrac{en}{m})-\tfrac{\beta}{40(1+B)}M\big]+\tfrac{13M}{40}\Big)\\
			&\le 
			\mathrm{P}_Z\Big(\sum_{i=1}^m Z_{[i]}^2 \ge  5m\log(\tfrac{en}{m}) 
			+\tfrac{M}{5}\Big) \le \exp\{-2M/25\}, \notag
		\end{align*}
		where we also used in the last step that 
		$\tfrac{M_0}{2c_3}-\tfrac{1+B}{\beta}-1=5$. Hence,
		\begin{align*}
			\mathrm{E}_{\theta} T_3 &= 
			\mathrm{E}_{\theta} 
			\sum_{I\in\mathcal{S}_M}\mathbbm{1}_{A_M}
			\hat{\pi}_I\big(\|\vartheta-\theta\|^2\ge M_0r^2(\theta)+ M\sigma^2 | X\big)\hat{\pi}(I|X) \\&
			\le  \exp\{-2M/25\}\mathrm{E}_{\theta}\sum_{I\in\mathcal{I}}\hat{\pi}(I|X) \le  \exp\{-2M/25\}.
		\end{align*}
		This completes the proof of assertion \eqref{th1_i} since, in view of \eqref{th1_T1}, \eqref{th1_T2} and 
		the last display,  we established  that
		$\mathrm{E}_{\theta} \hat{\pi}\big(\|\vartheta-\theta\|^2\ge M_0r^2(\theta)+ M\sigma^2  | X\big)
		\le \mathrm{E}_{\theta}(T_1+T_2+T_3) \le  (2+C_0)e^{-m_0 M}$, with constants 
		$M_0=2c_3(6+\tfrac{1+B}{\beta})$, $H_0=2+C_0$,
		$m_0= \min\{\frac{\beta}{8},\frac{c_2\beta}{40(1+B)}, \tfrac{2}{25}\}$
		and $C_0$ defined in \eqref{th1_fact1}. \medskip
		
		The proof of assertion \eqref{th1_ii} proceeds along similar lines.
		Recall the constants $c_1>2$, $c_2$, $c_3$ from Lemma \ref{lemma2} and
		define $M_1=4c_3(1+B+\beta)/\beta$. Introduce the subfamily of sets
		\[
		\bar{\mathcal{S}}_M=\bar{\mathcal{S}}_M(\theta)=
		\big\{ I \in \mathcal{I}: r^2(I,\theta) \le 2c_3 r^2(\theta)+\tfrac{\beta}{6(1+B)} M\sigma^2\big\},
		\]
		and the event 
		$\bar{A}_M=\bar{A}_M(\theta)=
		\big\{\sum_{i=1}^{\bar{m}} \xi_{[i]}^2 \le  \tfrac{(1+B)}{\beta}\bar{m}\log(\tfrac{en}{\bar{m}})
		+\frac{M}{6}\big\}$, where 
		$\bar{m}=\bar{m}_M(\theta)=\max\{|I|: I \in \bar{\mathcal{S}}_M\}$.
		Introduce the notation $\bar{\Delta}_M(\theta)=M_1r^2(\theta)+M\sigma^2$ for brevity.
		By the definition of $\hat{\theta}$ and the Cauchy-Schwartz inequality, 
		we have that $\|\hat{\theta}-\theta\|^2 \le \sum_{I \in \mathcal{I}}\|X(I)-\theta\|^2 \hat{\pi}(I|X)$,
		where $\|X(I)-\theta\|^2=\sigma^2\sum_{i\in I} \xi_i^2 + \sum_{i\in I^c} \theta_i^2$.
		Using this, we derive
		\begin{align*}
			&\mathrm{P}_{\theta}\big(\|\hat{\theta}-\theta\|^2 \ge  \bar{\Delta}_M(\theta) \big) \le
			\mathrm{P}_{\theta}\Big(\sum_{I \in \mathcal{I}}
			\|X(I)-\theta\|^2 
			\hat{\pi}(I|X) \ge \bar{\Delta}_M(\theta) \Big) \\  
			&\le 
			\mathrm{P}_{\theta}(\bar{A}^c_M)  +
			\mathrm{P}_{\theta}\Big(\Big\{\sum_{I \in \bar{\mathcal{S}}_M}
			\Big[\sigma^2\sum_{i\in I} \xi_i^2 + \sum_{i\in I^c} \theta_i^2\Big]
			\hat{\pi}(I|X) \ge\bar{\Delta}_M(\theta)/2\Big\} \cap \bar{A}_M\Big) \\
			&\qquad +
			\mathrm{P}_{\theta}\Big(\sum_{I \in \bar{\mathcal{S}}^c_M}
			\Big[\sigma^2\sum_{i\in I} \xi_i^2 + \sum_{i\in I^c} \theta_i^2\Big]
			\hat{\pi}(I|X) \ge \bar{\Delta}_M(\theta)/2\Big) =\bar{T}_1+\bar{T}_2+\bar{T}_3.
		\end{align*}
		
		Similar to \eqref{th1_T1}, we bound the term $\bar{T}_1$ 
		by Lemma \ref{lemma4} (see also \eqref{lem4_rel}):
		\[
		\bar{T}_1=\mathrm{P}_{\theta}(\bar{A}^c_M) = 
		\mathrm{P}\Big(\sum_{i=1}^{\bar{m}} \xi_{[i]}^2>\tfrac{(1+B)}{\beta}\bar{m}\log(\tfrac{en}{\bar{m}}) +\tfrac{M}{6}\Big) 
		\le \exp\big\{-M\beta/6\big\}.
		\]
		
		Now we evaluate the term $\bar{T}_2$.  
		Since $\bar{m}=\max\{|I|: I \in \bar{\mathcal{S}}_M\}$,  
		$\sigma^2\bar{m}\log(\tfrac{en}{\bar{m}}) \le 2c_3 r^2(\theta)+\tfrac{\beta}{6(1+B)} M\sigma^2$.
		Then for any $I \in \bar{\mathcal{S}}_M$, the event $\bar{A}_M$ implies that
		$\sum_{i\in I} \xi_i^2\le \sum_{i=1}^{\bar{m}} \xi_{[i]}^2 \le \tfrac{(1+B)}{\beta}\bar{m}\log(\tfrac{en}{\bar{m}}) 
		+\frac{M}{6}\le \tfrac{2c_3 (1+B)}{\beta} \tfrac{r^2(\theta)}{\sigma^2} + \tfrac{M}{3}$.  
		Also $\sum_{i \in I^c}\theta_{i}^2 \le r^2(I,\theta)\le 2c_3 r^2(\theta)+\tfrac{\beta}{6(1+B)} M\sigma^2$ 
		for any $I \in \bar{\mathcal{S}}_M$. Hence, for any $I \in \bar{\mathcal{S}}_M$, we obtain the implication
		\[
		\bar{A}_M \subseteq 
		\Big\{ \sigma^2 \sum_{i\in I} \xi_i^2 +\sum_{i \in I^c}\theta_{i}^2
		\le \tfrac{2c_3(1+B+\beta)}{\beta}r^2(\theta)+(\tfrac{1}{3}+\tfrac{\beta}{6(1+B)}) M\sigma^2\Big\}.
		\]
		As $M_1=4c_3(1+B+\beta)/\beta$, $\beta\in (0,1]$ and $B>0$, the last relation entails
		\begin{align*}
			\bar{T}_2 &=\mathrm{P}_{\theta}\Big(\Big\{\sum_{I \in \bar{\mathcal{S}}_M}
			\Big(\sigma^2 \sum_{i\in I} \xi_i^2 +\sum_{i \in I^c}\theta_{i}^2 \Big)\hat{\pi}(I|X) 
			\ge \tfrac{\bar{\Delta}_M}{2}\Big\} \cap \bar{A}_M\Big) \\
			&\le 
			\mathrm{P}_{\theta}\Big(\tfrac{2c_3(1+B+\beta)}{\beta}r^2(\theta)
			+(\tfrac{1}{3}+\tfrac{\beta}{6(1+B)})M\sigma^2 \ge \tfrac{M_1}{2} r^2(\theta)+ \tfrac{M}{2}\sigma^2\Big)=0.
		\end{align*}
		
		It remains to handle the term $\bar{T}_3$. 
		Applying first  the Markov inequality and then the Cauchy-Schwarz inequality, we obtain 
		\begin{align*}
			&\bar{T}_3 \le \frac{ \mathrm{E}_{\theta}  
				\big(\sum_{I \in \bar{\mathcal{S}}^c_M}
				\big[\sigma^2\sum_{i\in I} \xi_i^2 + \sum_{i\in I^c} \theta_{i}^2\big]
				\hat{\pi}(I|X) \big)}{\bar{\Delta}_M(\theta)/2}\\
			&\le 
			\frac{\sum_{I\in \bar{\mathcal{S}}^c_M} \big(\sigma^2
				\big[\mathrm{E}_{\theta}\big(\sum_{i\in I} \xi_i^2 \big)^2\big]^{1/2} 
				\big[ \mathrm{E}_{\theta}(\hat{\pi}(I|X))^2\big]^{1/2} 
				+r^2(I,\theta) \mathrm{E}_{\theta}\hat{\pi}(I|X)\big)}{\bar{\Delta}_M(\theta)/2} =T_{31} +T_{32}.
		\end{align*}
		
		For any $I \in \bar{\mathcal{S}}^c_M$, we have $c_3 r^2(\theta) 
		\le \tfrac{r^2(I,\theta)}{2} -\tfrac{\beta}{12(1+B)} M\sigma^2$, yielding the bound 
		\begin{align}
			\label{th1_rel11}
			\frac{c_2}{2}\big(r^2(I,\theta)-c_3 r^2(\theta)\big) &\ge
			C_1r^2(I,\theta)+C_2 M\sigma^2
			\quad \text{for any} \;\; I \in \bar{\mathcal{S}}^c_M,
		\end{align}
		where  $C_1= c_2/4$ and $C_2=c_2\beta/[24(1+B)]$. 
		By \eqref{th1_rel11} and Lemma \ref{lemma2},
		\begin{align}
			\label{th1_rel12}
			\big[\mathrm{E}_{\theta}\hat{\pi}(I|X) \big]^{1/2}\le  
			\big(\tfrac{ne}{|I|}\big)^{-c_1|I|/2} 
			\exp\big\{-C_1\sigma^{-2}r^2(I,\theta)-C_2M\big\}
			\quad \text{for any} \;\; I \in \bar{\mathcal{S}}^c_M.
		\end{align}
		Since $c_1>2$, \eqref{th1_fact1} gives  
		$\sum_{I\in \mathcal{I}}\big(\tfrac{ne}{|I|}\big)^{-c_1|I|/2} \le 
		(1-e^{-c_1/2})^{-1}\triangleq C_3$.
		According to \eqref{moment_bound} with $\rho=\min\{C_1,B/2\}$,
		$\Big[\mathrm{E}\big(\sum_{i\in I} \xi_i^2 \big)^2\Big]^{1/2}
		\le \frac{B}{\beta\rho}\exp\{\rho |I|\}$.  
		If $M\in[0,1]$, the claim (ii) holds for any $H_1\ge e^{m_1}$.
		Let  $M\ge1$, then $\sigma^2/\bar{\Delta}_M(\theta) \le M^{-1}\le 1$. 
		Besides, $\sigma^{-2}r^2(I,\theta) \ge |I|\log(en/|I|) \ge |I|$. 
		Piecing all these relations together with \eqref{th1_rel12}, we derive
		\begin{align*}
			T_{31} &\le \tfrac{2B}{\beta\rho}
			\sum_{I\in \bar{\mathcal{S}}^c_M}\exp\{\rho |I|\}
			\big(\tfrac{ne}{|I|}\big)^{-c_1|I|/2} 
			\exp\big\{-C_1\sigma^{-2}r^2(I,\theta)-C_2M\big\} 
			\le C_4\exp\{-C_2M\},
		\end{align*}
		where $C_4=2BC_3/(\beta\rho) =2BC_3/(\beta\min\{C_1,B\})$.
		Finally, by \eqref{th1_fact1}, \eqref{th1_rel12} and the facts that 
		$\max_{x\ge 0}\{xe^{-cx} \}\le (ce)^{-1}$ (for any $c>0$) and 
		$\sigma^2/\bar{\Delta}_M(\theta) \le 1$, we bound the term $T_{32}$:
		\begin{align*}
			T_{32}&=\tfrac{2}{\bar{\Delta}_M(\theta)} \sum_{I\in \bar{\mathcal{S}}^c_M} 
			r^2(I,\theta)   \mathrm{E}_{\theta}\hat{\pi}(I|X) \\
			&\le 
			\tfrac{2}{\bar{\Delta}_M(\theta)} \sum_{I\in \bar{\mathcal{S}}^c_M} 
			r^2(I,\theta)  
			\big(\tfrac{ne}{|I|}\big)^{-c_1|I|}\exp\big\{-2C_1\sigma^{-2}r^2(I,\theta)-2C_2M\big\} 
			\le C_5 \exp\{-2C_2 M\},
		\end{align*}
		where $C_5=C_0/(C_1 e)$.
		The assertion \eqref{th1_ii} is proved since we showed that 
		$\mathrm{P}_{\theta} \big(\|\hat{\theta}-\theta\|^2\ge M_1r^2(\theta)+ M\sigma^2\big)
		\le H_1 e^{-m_1 M}$ with $M_1=4c_3(1+B+\beta)/\beta$, 
		$H_1=(1+C_4+C_5)\vee e^{m_1}$,  $m_1= \min\{\frac{\beta}{6},C_2\}$.
	\end{proof}
	
	\begin{proof}[Proof of Theorem \ref{th2}]
		First we prove (i). If the  inequality $|I \backslash I_o|\log(\tfrac{en}{|I|})<\sum_{i\in I \backslash I_o}\tfrac{\theta_i^2}{\sigma^2}$ would hold for some $I\in\mathcal{I}$, then 
		\begin{align*}
			r^2(I\cup I_o,\theta) 
			&= \sum_{i \not\in I\cup I_o} \theta^2_i+\sigma^2 |I\cup I_o| 
			\log(\tfrac{en}{|I\cup I_o|}) \le \sum_{i \not\in I\cup I_o} \theta^2_i 
			+\sigma^2 |I \backslash I_o|\log(\tfrac{en}{|I|})+\sigma^2 |I_o|\log(\tfrac{en}{|I_o|}) \\
			&< 
			\sum_{i \not\in I\cup I_o} \theta^2_i+\sum_{i\in I \backslash I_o}\tfrac{\theta_i^2}{\sigma^2}
			+ \sigma^2 |I_o| \log(\tfrac{en}{|I_o|})
			= \sum_{i \not\in I_o} \theta^2_i+\sigma^2 |I_o| 
			\log(\tfrac{en}{|I_o|})=r^2(\theta),
		\end{align*}
		which contradicts the definition of the oracle. Hence, 
		$\sum_{i\in I \backslash I_o}\tfrac{\theta_i^2}{\sigma^2}\le
		|I \backslash I_o| \log(\tfrac{en}{|I|})$ for any $I \in \mathcal{I}$.
		Define $c_4=\varkappa\beta-\tfrac{\beta}{2}-B-1$ and 
		note that $c_4>1$ by the condition of the theorem.
		Using the relation $\sum_{i\in I \backslash I_o}\tfrac{\theta_i^2}{\sigma^2}
		\le|I \backslash I_o| \log(\tfrac{en}{|I|})\le 
		|I| \log(\tfrac{en}{|I|})$ and Lemma \ref{lemma1} with $h=\beta$ and
		$I_0=I_o\cap I$ (so that $I\backslash I_0=I\backslash I_o$), 
		we obtain for each $I\in\mathcal{G}_1=\{I\in\mathcal{I}:|I| \log(\frac{en}{|I|}) \ge 
		M'_0 |I_0|\log(\frac{en}{|I_0|})+M\}$ with $M'_0=\varkappa\beta+\tfrac{\beta}{2}$,
		\begin{align*}
			\mathrm{E}_\theta\hat{\pi}(I|X) 
			&\le\big[\tfrac{\lambda_{I}}{\lambda_{I_0}}\big]^\beta
			\exp\Big\{\beta \sum_{i\in I\backslash I_0}
			\tfrac{\theta_i^2}{\sigma^2} 
			+\tfrac{\beta}{2}|I_0| \log(\tfrac{en}{|I_0|})-(\tfrac{\beta}{2}-B)|I| \log(\tfrac{en}{|I|})\Big\}\\
			&\le
			(\tfrac{ne}{|I|})^{-c_4 |I|} 
			\exp\big\{-(\varkappa\beta-\tfrac{\beta}{2} -B -c_4)|I|\log(\tfrac{en}{|I|})
			+(\beta \varkappa +\tfrac{\beta}{2}) |I_0|\log(\tfrac{en}{|I_0|})\big\}\\
			&=
			(\tfrac{ne}{|I|})^{-c_4|I|} \exp\big\{-|I| \log(\tfrac{en}{|I|}) 
			+(\beta \varkappa +\tfrac{\beta}{2})|I_0| \log(\tfrac{en}{|I_0|})\big\}\\
			&\le 
			(\tfrac{ne}{|I|})^{-c_4|I|}\exp\big\{-(M'_0-\varkappa\beta-\tfrac{\beta}{2})
			|I_0| \log(\tfrac{en}{|I_0|})-M\big\}
			=(\tfrac{ne}{|I|})^{-c_4|I|} 
			e^{-M}.
		\end{align*}
		Since $c_4>1$, by the same reasoning as in \eqref{sum_lambda}
		we bound $\sum_{I\in\mathcal{I}} (\tfrac{ne}{|I|})^{-c_4|I|}\le(1-e^{1-c_4})^{-1}\triangleq H'_0$.
		Using this and the last display, we finish the proof of (i):
		\begin{align*}
			\mathrm{E}_\theta\hat{\pi}(I \in\mathcal{G}_1| X) &=
			\sum_{I\in\mathcal{G}_1}\mathrm{E}_\theta\hat{\pi}(I|X)\le 
			e^{-M}\sum_{I\in\mathcal{I}} (\tfrac{ne}{|I|})^{-c_4|I|}
			\le H'_0 e^{-M}.
		\end{align*}

		Next we prove (ii). 
		Define $\mathcal{G}_2=\mathcal{G}_2(I')
		=\{I\in\mathcal{I}: \sum_{i\in I'\backslash I} \frac{\theta^2_i}{\sigma^2} 
		\ge  \bar{\tau} |I\cup I'| \log(\frac{en}{|I\cup I'|})+M\}$.
		Using 
		\eqref{prior_lambda} and Lemma \ref{lemma1} with $h=\beta$ 
		and $I_0=I_0(I, \theta)=I\cup I'$,  
		we evaluate  for each $I\in \mathcal{G}_2$
		\begin{align*}
			\mathrm{E}_{\theta}
			&\hat{\pi}(I|X)\le
			\big[\tfrac{\lambda_I}{\lambda_{I_0}}\big]^\beta
			\exp\Big\{-\tfrac{\beta}{3}\sum_{i\in I_0\backslash I}\tfrac{\theta_i^2}{\sigma^2}
			+(\tfrac{\beta}{2}+B)|I_0|\log\big(\tfrac{en}{|I_0|}\big)
			-\tfrac{\beta}{2}|I|\log\big(\tfrac{en}{|I|}\big)\Big\}\notag\\
			&=
			\big[\tfrac{\lambda_I}{c_{\varkappa,n}}\big]^\beta
			\exp\Big\{-\tfrac{\beta}{3}\sum_{i\in I'\backslash I}\tfrac{\theta_i^2}{\sigma^2}
			+(\varkappa \beta +\tfrac{\beta}{2}+B) |I\cup I'|\log \big(\tfrac{en}{|I\cup I'|}\big)
			-\tfrac{\beta}{2}|I|\log\big(\tfrac{en}{|I|}\big)\Big\}\notag\\
			&\le
			\big[\tfrac{\lambda_I}{c_{\varkappa,n}}\big]^{\beta+\tfrac{\beta}{2\varkappa}} 
			\exp\big\{\big(-\tfrac{\beta}{3}\bar{\tau}+\varkappa\beta+\tfrac{\beta}{2}+B\big)
			|I\cup I'|\log \big(\tfrac{en}{|I\cup I'|}\big)
			-\tfrac{\beta}{3}M\big\}
			\le \big[\tfrac{\lambda_I}{c_{\varkappa,n}}\big]^{\beta+\tfrac{\beta}{2\varkappa}} 
			e^{-\tfrac{\beta}{3}M}.
		\end{align*}
		Since $\varkappa>\beta^{-1}-\tfrac{1}{2}$, by the same reasoning as in \eqref{sum_lambda} 
		we bound
		$\sum_{I} \big[\tfrac{\lambda_I}{c_{\varkappa,n}}\big]^{\beta(1+1/2\varkappa)}
		\le (1-e^{1-\varkappa\beta-\beta/2})^{-1}\triangleq H'_1$.
		This relation and the last display imply claim (ii):
		with $m'_0= \tfrac{\beta}{3}$,
		\begin{align}
			\label{th2_iia}
			\mathrm{E}_{\theta} \hat{\pi}\big(I\in \mathcal{G}_2|X\big)
			=\sum_{I\in \mathcal{G}_2} \mathrm{E}_{\theta}\hat{\pi}(I|X)
			\le H'_1\exp\{
			-m'_0M\}.
		\end{align}
		Let us derive the second claim of (ii). If $|I| \log(\tfrac{en}{|I|})\le 
		\varrho |I_*| \log(\tfrac{en}{|I_*|}) - M$, then 
		$|I\cup I_*|\log(\tfrac{en}{|I \cup I_*|})\le |I|\log(\tfrac{en}{|I|})
		+|I_*|\log(\tfrac{en}{|I_*|})\le (1+\varrho) |I_*|\log(\tfrac{en}{|I_*|}) -M$.
		Hence, 
		$|I_*|\log(\tfrac{en}{|I_*|}) \ge \tfrac{1}{1+\varrho} |I\cup I_*|
		\log(\tfrac{en}{|I \cup I_*|})+\tfrac{M}{1+\varrho}$, which, together with 
		the definition of the $\tau$-oracle, imply
		\begin{align}
			\sum_{i\in I_*\backslash I} \tfrac{\theta^2_i}{\sigma^2} 
			&\ge\Big(\sum_{i\in I^c} \tfrac{\theta^2_i}{\sigma^2} 
			-\sum_{i\in I_*^c} \tfrac{\theta^2_i}{\sigma^2} \Big)
			\ge\tau_0\big( |I_*|\log(\tfrac{en}{|I_*|})-|I|\log(\tfrac{en}{|I|})\big) \notag\\
			&\ge\tau_0(1-\varrho)|I_*|\log( \tfrac{en}{|I_*|}) + \tau M
			\ge \tau' |I\cup I_*|\log(\tfrac{en}{|I \cup I_*|})
			+\tfrac{2\tau_0}{1+\varrho} M,
			\label{proof_th2_ii}
		\end{align}
		where $\tau'=\tfrac{1-\varrho}{1+\varrho} \tau_0> \bar{\tau}$
		by the condition of the theorem. Thus, we obtain 
		\[
		\mathrm{E}_{\theta}\hat{\pi}\big(I: |I| \log(\tfrac{en}{|I|})\le 
		\varrho |I_*| \log(\tfrac{en}{|I_*|}) - M \big|X\big)  \le 
		\mathrm{E}_{\theta}\hat{\pi}\Big(\sum_{i\in I_*\backslash I} \tfrac{\theta^2_i}{\sigma^2} 
		\ge \tau'|I\cup I_*| \log(\tfrac{en}{|I\cup I_*|})+\tfrac{2\tau_0}{1+\varrho}M |X\Big).
		\]
		By this and \eqref{th2_iia} with $I'=I_*$, the second claim of (ii) follows  with
		$\alpha'_1 = \tau'-\bar{\tau}>0$ and $m'_1=\tfrac{2\tau_0 m'_0}{1+\varrho}$.

		Finally, let us prove (iii).
		Denote $\mathcal{G}_3=\mathcal{G}_3(\theta,M)=\{I: \, r^2(I,\theta) \ge c_3 r^2(\theta) + M \sigma^2\}$, 
		where the constants $c_1>2$, $c_2$, $c_3$ 
		are defined in Lemma \ref{lemma2}. Applying Lemma \ref{lemma2} 
		and using the fact \eqref{th1_fact1}, 
		we complete the proof of (iii):
		\begin{align*}
			&\mathrm{E}_{\theta}\hat{\pi}\big(I \in\mathcal{G}_3\big| X\big)=
			\sum_{I\in\mathcal{G}_3}\mathrm{E}_{\theta}\hat{\pi}(I|X)  \le
			e^{-c_2M} \sum_{I\in\mathcal{I}} (\tfrac{ne}{|I|})^{-c_1 |I|}  \le 
			C_0 e^{-c_2 M}. \qedhere
		\end{align*}
	\end{proof}

	\begin{proof}[Proof of Theorem \ref{th8}] 
		The biggest part of the proof is already contained in Theorem \ref{th2}. 
		We first establish the coverage property. 
		The constants $M_1$, $H_1$ and $m_1$  are defined in Theorem \ref{th1},  
		the constant $\varrho$ is from \eqref{def_I_*}.
		Take some $M_2>\tfrac{M_1}{\varrho}$, for example $M_2=\tfrac{M_1}{\varrho}+1$. 
		From  \eqref{oracle} and \eqref{def_b}, it follows that 
		$r^2(\theta)\le r^2(I_*,\theta)=(b(\theta)+1)\sigma^2 |I_*
		|\log(\tfrac{en}{|I_*|})+b(\theta)\sigma^2\le(b(\theta)+1)
		\sigma^2(|I_*|\log(\tfrac{en}{|I_*|})+1)$.
		Combining this with claims (ii) from Theorems \ref{th1} and \ref{th2} 
		and the definition \eqref{check_r} of $\hat{r}$ yields the coverage property:
		\begin{align*}
			&\mathrm{P}_{\theta}\big(\theta  \notin B(\hat{\theta},[(b(\theta)+1)M_2\hat{r}^2
			+(b(\theta)+2)M\sigma^2]^{1/2}\big) \\
			&\le\mathrm{P}_{\theta}\Big(\|\hat{\theta}-\theta\|^2>(b(\theta)+1)M_2\hat{r}^2
			+(b(\theta)+2)M\sigma^2, \hat{r}^2\ge \varrho\sigma^2|I_*|\log(\tfrac{en}{|I_*|})
			+\sigma^2-\tfrac{M\sigma^2}{M_2}\Big) \\
			&\qquad +\mathrm{P}_{\theta}\Big(\hat{r}^2<\varrho \sigma^2|I_*|\log(\tfrac{en}{|I_*|})
			+\sigma^2-\tfrac{M\sigma^2}{M_2}\Big)\notag\\
			&\le \mathrm{P}_{\theta}\Big(\|\hat{\theta}-\theta\|^2>  \varrho M_2 r^2(\theta)+M\sigma^2\Big)
			+ \mathrm{P}_{\theta}\Big(|\hat{I}|\log(\tfrac{en}{|\hat{I}|})< 
			\varrho|I_*|\log(\tfrac{en}{|I_*|})-\tfrac{M}{M_2}\Big)\\ 
			&\le H_1\big(\tfrac{en}{ |I_o|}\big)^{-\alpha_1|I_o|} e^{-m_1M}
			+H'_1 \big(\tfrac{en}{ |I_*|}\big)^{-\alpha'_1|I_*|}e^{-m''_1M}
			\le  H_2 e^{-m_2M},
		\end{align*}
		where $\alpha_1= \varrho M_2 -M_1$, $m''_1= m'_1/M_2$, $H_2=H_1+H'_1$, 
		$m_2=m_1\wedge m''_1$; $H'_1, \alpha'_1, m'_1$ are defined in Theorem \ref{th2} 
		and the constant $\varrho$ is from \eqref{def_I_*}.
		The first claim of the theorem follows.
		
		The size property follows from the definition \eqref{check_r} of 
		$\hat{r}$,  Remark \ref{rem1a} and property (i) of Theorem  \ref{th2}.
		Indeed, $\mathrm{P}_{\theta}\big(\hat{r}^2\ge 
		\sigma^2 (M'_0+\alpha) |I_o | \log(\tfrac{en}{|I_o|})+(M+1)\sigma^2\big)
		=\mathrm{P}_{\theta}\big(|\hat{I}| \log(\tfrac{en}{|\hat{I}|})\ge  
		(M'_0+\alpha) |I_o | \log(\tfrac{en}{|I_o|})+M\big)\le 
		\mathrm{P}_{\theta}\big(|\hat{I}| \log(\tfrac{en}{|\hat{I}|})\ge  
		M'_0 |I\cap I_o | \log(\tfrac{en}{|I\cap I_o|})+ \alpha| I_o | \log(\tfrac{en}{|I_o|})+M\big)
		\le H'_0 (\tfrac{ne}{|I_o|})^{-\alpha |I_o|} e^{-M}$.
	\end{proof}
	
	\begin{proof}[Proof of Theorem \ref{th4}] 
		Since $X'=\theta+\xi'$, we rewrite \eqref{radius2} as 
		\begin{align}
			\tilde{R}^2_M&=\big(\|X'-\hat{\theta}\|^2 -\sigma^2\mathrm{E} \|\xi'\|^2  
			+2\sigma^2G_M\sqrt{n}\big)_+ \notag\\
			&=\big(\|\theta-\hat{\theta}\|^2+\sigma^2\big(\|\xi'\|^2-\mathrm{E} \|\xi'\|^2\big)+2\sigma 
			\langle\xi',(\theta-\hat{\theta})\rangle+2\sigma^2G_M\sqrt{n}\big)_+.
			\label{R_M}
		\end{align}
		
		Introduce the events $\mathcal{C}_M=\big\{2|\langle\xi',(\theta-\hat{\theta})\rangle|<
		\sqrt{M(M_1r^2(\theta) +M\sigma^2)} \big\}$ and $\mathcal{D}_M=
		\big\{\|\hat{\theta}-\theta\|^2<M_1r^2(\theta) +M\sigma^2\big\}$. 
		According to \eqref{cond_A3}, $\hat{\theta}$ and $\hat{I}$ are independent of $\xi'$.
		Using this fact, the first relation from \eqref{cond_A3} and Theorem \ref{th1}, 
		we obtain that
		\begin{align}
			\label{expression}
			\mathrm{P}_\theta(\mathcal{C}_M^c)&=
			\mathrm{E}_\theta\mathrm{P}_\theta(\mathcal{C}_M^c\cap\mathcal{D}_M|X)
			+\mathrm{P}_\theta(\mathcal{C}_M^c\cap\mathcal{D}^c_M)\notag\\
			&\le \mathrm{E}_\theta \Big[\psi_1\big(\tfrac{M(M_1r^2(\theta)+M\sigma^2)}{4\|\hat{\theta}-\theta\|^2}\big)
			\mathbbm{1}_{\mathcal{D}_M}\Big]+\mathrm{P}_\theta(\mathcal{D}_M^c)\le 
			\psi_1(M/4)+H_1e^{-m_1M}.
		\end{align}
		Since $r^2(\theta)\le r^2([n], \theta)=\sigma^2n$ 
		by the oracle definition \eqref{oracle}, the event $\mathcal{C}_M$ implies 
		that $2\sigma\langle\xi',(\theta-\hat{\theta})\rangle>
		-\sigma\sqrt{M(M_1 \sigma^2n +M\sigma^2)} \ge -\sigma^2 G_M \sqrt{n}$.
		Combining this with \eqref{R_M}, \eqref{expression} and the second relation from \eqref{cond_A3}
		yields the coverage relation:
		\begin{align*}
			&\mathrm{P}_\theta \big(\theta  \notin B(\hat{\theta},\tilde{R}_M)\big)
			=\mathrm{P}_\theta\big(\theta  \notin B(\hat{\theta},\tilde{R}_M),\mathcal{C}_M\big)
			+\mathrm{P}_\theta\big(\theta\notin B(\hat{\theta},\tilde{R}_M),\mathcal{C}_M^c\big)\\
			&\le \mathrm{P}_\theta \big(\|\theta-\hat{\theta}\|^2\ge \tilde{R}^2_M,\mathcal{C}_M\big)
			+\mathrm{P}_\theta(\mathcal{C}_M^c)  
			\le\mathrm{P}_\theta\big(0\ge\sigma^2(\|\xi'\|^2-\mathrm{E} \|\xi'\|^2)+\sigma^2G_M \sqrt{n}\big)+\mathrm{P}_\theta(\mathcal{C}_M^c)\\
			&\le\mathrm{P}_\theta\big(\|\xi'\|^2-\mathrm{E} \|\xi'\|^2\le-M\sqrt{n} \big)
			+\psi_1(M/4)+H_1e^{-m_1M}  
			\le \psi_2(M) +\psi_1(M/4)+H_1e^{-m_1M}.
		\end{align*}
		
		Let us show the size property. By \eqref{expression}, 
		$\mathrm{P}_\theta\big(2\sigma\langle\xi',(\theta-\hat{\theta})\rangle\ge 
		\sigma^2G_M\sqrt{n}\big) \le
		\mathrm{P}_\theta\big( 2\langle\xi',(\theta-\hat{\theta})\rangle>
		\sqrt{M(M_1r^2(\theta) +M\sigma^2)}\le \mathrm{P}_\theta(\mathcal{C}^c_M)
		\le\psi_1(M/4)+H_1e^{-m_1M}$. This, Theorem \ref{th1} and \eqref{R_M} imply
		\begin{align*}
			&\mathrm{P}_\theta\big(\tilde{R}_M^2 \ge g_M(\theta,n)\big)\le 
			\mathrm{P}_\theta \big( \|\theta-\hat{\theta}\|^2 \ge M_1  r^2(\theta)+M\sigma^2\big) +
			\mathrm{P}_\theta\big(\sigma^2\big(\|\xi'\|^2-\mathrm{E} \|\xi'\|^2\big) 
			\ge \sigma^2G_M\sqrt{n} \big) \\
			& \qquad
			+\mathrm{P}_\theta\big(2\sigma\langle\xi',(\theta-\hat{\theta})\rangle\ge 
			\sigma^2G_M\sqrt{n}\big)\le H_1e^{-m_1M} +\psi_2(M)+\psi_1(M/4)+H_1e^{-m_1M}.
			\qedhere
		\end{align*}
	\end{proof}

	\begin{proof}[Proof of Theorem \ref{cor1_dimension}]
		Observe that  $r^2(\theta) \le  r^2(I^*(\theta),\theta) 
		\le \sigma^2 s(\theta) \log (\frac{en}{s(\theta)})$. Since the function 
		$x\mapsto x\log (en/x)$ is increasing over $(0,n]$, $|I|\ge M s(\theta)$ 
		implies that $r^2(I,\theta)\ge  \sigma^2 |I| \log(\frac{en}{|I|}) \ge \sigma^2 M s(\theta) \log(\frac{en}{Ms(\theta)})$.
		Thus, if $|I|\ge M s(\theta)$, then
		\begin{align*} 
			r^2(I,\theta) \ge M\sigma^2 s(\theta) \log(\tfrac{en}{Ms(\theta)})\ge 
			M_4r^2(\theta)-M_4\sigma^2 s(\theta) \log(\tfrac{en}{s(\theta)})
			+ M\sigma^2 s(\theta) \log(\tfrac{en}{Ms(\theta)}).
		\end{align*}
		The first claim follows from Theorem \ref{th2} with $M_4=c_3$ and $m_4=c_2$.
		
		To prove the second claim, note that for any $M'>2M_4$,  $|I|\ge M' s(\theta)$ implies that 
		\begin{align*}
			r^2(I,\theta)&\ge \sigma^2|I| \log (en/|I|)\ge 
			M' \sigma^2s(\theta)[\log (en/s(\theta))-\log M'] 
			\ge 
			\frac{M'}{2} \sigma^2s(\theta)\log (en/s(\theta)),
		\end{align*} 
		provided that $s(\theta)<en/(M')^2$. Since $r^2(\theta) \le r^2(I^*(\theta), \theta) \le 
		\sigma^2 s(\theta) \log(en/s(\theta))$, the relation above implies that $r^2(I,\theta)
		\ge M_4 r^2(\theta)+M\sigma^2$, where $M=(M'/2-M_4)s(\theta)\log(en/s(\theta))$. 
		Hence by Theorem \ref{th2}, the assertion holds for $M'_4=M'$  whenever $s(\theta)<en/(M')^2$. 
		If $s(\theta)\ge en/(M')^2$, the result trivially holds by choosing $M'_4=(M')^2/e$. 
		Hence the choice $M'_4\ge\max\{M', (M')^2/e\}$ ensures the result with 
		$m'_4=m_4(M'/2-M_4)$ for any $\theta\in\mathbb{R}^n$.  
	\end{proof}

	\begin{proof}[Proof of Theorem \ref{cor2_dimension}]
		Recall \eqref{weak_balls_minimax_risk}: 
		$\sigma^{-2}r^2(\theta) \le  
		Kn (\tfrac{p_n}{n\sigma})^s\big[\log (\tfrac{n\sigma}{p_n})\big]^{1-s/2}
		$ for each 
		$\theta\in m_s[p_n]$ with some $K=K(s)$. On the other hand, if  $|I|\!>\!Mp_n^*=
		Men(\tfrac{p_n}{n\sigma})^s \big[\log(\tfrac{n\sigma}{p_n})\big]^{-s/2}$, then 
		$\sigma^{-2}r^2(I,\theta)\ge|I|\log(\tfrac{en}{|I|})\ge  M p_n^*\log(\tfrac{en}{Mp^*_n})
		=M p_n^*\big[s \log (\tfrac{n\sigma}{p_n})+\frac{s}{2} \log \log (\tfrac{n\sigma}{p_n})
		-\log(M)\big] \ge Msp_n^*\log (\tfrac{n\sigma}{p_n})$ for sufficiently large $n$ as $p_n =o(n)$.
		Then, for any $\theta\in m_s[p_n]$, $M> c_3 K/(se)$ and $|I|>
		Mp_n^*$, we have that, for sufficiently large $n$,  
		\begin{align*}
			\sigma^{-2}\big(r^2(I,\theta)-c_3 r^2(\theta)\big)&\ge M p_n^*\log(\tfrac{en}{Mp^*_n}) 
			-c_3 K
			n (\tfrac{p_n}{n\sigma})^s\big[\log (\tfrac{n\sigma}{p_n})\big]^{1-s/2}
			\\&\ge Ms p_n^*\log(\tfrac{n\sigma}{p_n})-c_3Ke^{-1}
			p_n^*\log (\tfrac{n\sigma}{p_n})
			=(Ms-c_3Ke^{-1})  p_n^*\log(\tfrac{n\sigma}{p_n}).
		\end{align*}
		
		Finally, applying Theorem \ref{th2}, we obtain
		\begin{align*}
			\sup_{\theta\in m_s[p_n]} \mathrm{E}_{\theta}\hat{\pi}(I: |I|>Mp_n^*|X) 
			&\le C_0 \exp\big\{-c_2s \big(M-c_3K(se)^{-1}\big)p_n^*\log (\tfrac{n\sigma}{p_n})\big\}, 
		\end{align*}
		which gives the claim with  $m_5=c_2s$ and $M_5=c_3K(se)^{-1}$.
	\end{proof}

\end{document}